\documentclass{article}
\usepackage{amsmath,amssymb,amsthm,graphicx,epsfig,float,url}
\usepackage[colorlinks=true,citecolor={Plum},linkcolor={Periwinkle}]{hyperref}
\usepackage{pdfsync}
\usepackage[usenames, dvipsnames]{xcolor}
\usepackage{tikz}
\usepackage{subfig}
\usepackage{bbm}
\usepackage{stmaryrd}
\usepackage{amssymb}
\usepackage{mathrsfs}  
\usepackage{caption}
\usepackage{float}
\usepackage{esint}

\usepackage{dsfont}
\usepackage[makeroom]{cancel}
\usetikzlibrary{patterns}
\newcommand{\R}{\textnormal{I\kern-0.21emR}}
\newcommand{\N}{\textnormal{I\kern-0.21emN}}

\renewcommand{\geq}{\geqslant}
\renewcommand{\leq}{\leqslant}

\def\B{{\mathbb B}}

\def\e{{\varepsilon}}

\def\tmm{{\theta_{m,\mu}}}
\def\dtmm{{\dot\theta_{m,\mu}}}
\def\ddtmm{{\ddot\theta_{m,\mu}}}
\def\pmm{{p_{m,\mu}}}
\def\umm{{u_{m,\mu}}}

\allowdisplaybreaks

\def\YYint#1#2#3{{\setbox0=\hbox{$#1{#2#3}{\iint}$}
    \vcenter{\hbox{$#2#3$}}\kern-.51\wd0}}
 
\def\n{{\nabla}}
\def\p{{\varphi}}

\usepackage{xargs}% use more than one optional parameter in a new command
 \usepackage[colorinlistoftodos,textsize=small]{todonotes}
 \newcommandx{\unsure}[2][1=]{\todo[linecolor=red,backgroundcolor=red!25,bordercolor=red,#1]{#2}}
 \newcommandx{\change}[2][1=]{\todo[linecolor=blue,backgroundcolor=blue!25,bordercolor=blue,#1]{#2}}
 \newcommandx{\info}[2][1=]{\todo[linecolor=green,backgroundcolor=green!25,bordercolor=green,#1]{#2}}
 \newcommandx{\improvement}[2][1=]{\todo[linecolor=yellow,backgroundcolor=yellow!25,bordercolor=yellow,#1]{#2}}
 
  \newcommandx{\biblio}[2][1=]{\todo[linecolor=blue,backgroundcolor=magenta!25,bordercolor=blue,#1]{#2}}

\usepackage[dvipsnames]{xcolor}

\newtheorem*{theorem*}{Theorem}

\newtheorem{theorem}{Theorem}

\newtheorem{material}{material}
\newtheorem{proposition}[material]{Proposition}

\newtheorem{lemma}[material]{Lemma}

\newtheorem{remark}[material]{Remark}

\def\O{{\Omega}}
 
 \numberwithin{equation}{section}

\begin{document}
%\nocite{*}
\title{{Optimisation of the total population size for logistic diffusive equations: bang-bang property and fragmentation rate}}% for low diffusivities }

%    Remove any unused author tags.

%    author one information
\author{Idriss Mazari\footnote{Technische Universit\"{a}t Wien, Institute of Analysis and Scientific Computing, 8-10 Wiedner Haupstrasse, 1040 Wien (\texttt{idriss.mazari@tuwien.ac.at})} \and Gr\'egoire Nadin\footnote{ CNRS, Sorbonne Universit\'es, UPMC Univ Paris 06, UMR 7598, Laboratoire Jacques-Louis Lions, F-75005, Paris, France (\texttt{gregoire.nadin@sorbonne-universite.fr})}
	\and Yannick Privat\footnote{Universit\'e de Strasbourg, CNRS UMR 7501, INRIA, Institut de Recherche Math\'ematique Avanc\'ee (IRMA), 7 rue Ren\'e Descartes, 67084 Strasbourg, France ({\tt yannick.privat@unistra.fr}).}
}
\date{}
%\address{}
%\curraddr{
%\email{idriss.mazari@upmc.fr}
%\th

\maketitle

\begin{abstract}
In this article, we give an in-depth analysis of the problem of optimising the total population size for a standard logistic-diffusive model. This optimisation problem stems from the study of spatial ecology and amounts to the following question: assuming a species evolves in a domain, what is the best way to spread resources in order to ensure a maximal population size at equilibrium? {In recent years, many authors contributed to this topic.} We settle here the proof of two fundamental properties of optimisers: the bang-bang one %(related to patch models), 
which had so far only been proved under several strong assumptions, and the other one is the fragmentation of maximisers. Here, we prove the bang-bang property in all generality using a new spectral method. 
%Bearing in mind that the equation we consider writes $\mu \Delta u+um-u^2=0$ and that $m$ is the control to be optimised under $L^1-L^\infty$ constraints, we give, in the conclusion, \textcolor{red}{detailed} comments explaining how one may adapt the technique to obtain the bang-bang property for 
{The technique introduced to demonstrate the bang-bang character of optimizers can be adapted and generalized to many optimization problems with other classes of bilinear optimal control problems where the state equation is semilinear and elliptic. We comment on it in a conclusion section.}
Regarding the geometry of maximisers, we exhibit a blow-up rate for the $BV$-norm of maximisers as the diffusivity gets smaller: if $\O$ is an orthotope and if $m_\mu$ is an optimal control, then $\Vert m_\mu\Vert_{BV}\gtrsim \sqrt{\mu}$. The proof of this results relies on a very fine energy argument.  
%Finally, we provide numerical simulations that illustrate our results.

\end{abstract}


\medskip

\paragraph{Acknowledgment.}I. Mazari and Y. Privat were partiallly supported by the French ANR Project ANR-18-CE40-0013 - SHAPO on Shape Optimization. I Mazari was partially supported by the Austrian Science Fund (FWF) projects I4052-N32 and F65.  I. Mazari, G. Nadin and Y. Privat were partially supported by the Project "Analysis and simulation of optimal shapes - application to lifesciences" of the Paris City Hall.

%\tableofcontents

\section{Introduction}
%The starting point of this article is a calculus of variations problem that stems from the study of mathematical biology.
{This article is devoted to the study of a problem of calculus of variations motivated by questions of spatial ecology. }
This problem is related to the ubiquitous question of \emph{optimal location of resources}. While we further specify what we mean by \textquotedblleft optimal\textquotedblright\,  in {what follows}, let us note that optimisation problems related to the location of resources are a possible way to tackle the question of \emph{spatial heterogeneity} in reaction-diffusion equations. In this context, spatial heterogeneity is interpreted as heterogeneity of the resources available to a given population.

 In this paper, {we thoroughly analyse the issue of optimising the total population size with respect to the resource distribution.}
% , and we give a complete analysis of this problem. 
% The methods introduced and used are then developed in the context of time-discretized parabolic bilinear optimal control problems{*** pas compris cette phrase, car finalement, cela ne figure pas dans le papier, non ?***}. \
{The reaction-diffusion model we deal with is made precise in Section~\ref{Se:Model} and the precise statement of our main results in Section~\ref{Se:Elliptic}.}
{In a nutshell,} our results may be recast as follows: 
\begin{itemize}
\item First, we give a characterisation of pointwise properties of optimal resource distributions (also called the \emph{bang-bang property}) that has been partially tackled in \cite{Mazari2020,NagaharaYanagida}; in these previous contributions, the contents of which we discuss in Sections~\ref{Se:Model} and \ref{Se:Bib}, {partial answers are provided under several technical assumptions}. {We present here a new method that we believe to be flexible and versatile enough to be applied to a wide class of \emph{bilinear optimal control problem}}, and that provides a positive answer to the question of knowing whether optimal resource distributions are bang-bang.
\item Second, we prove a \emph{fragmentation phenomenon}, with explicit blow-up rates: as has been noticed \cite{LouNagaharaYanagida,Mazari2020,MRBSIAP}, for the optimisation of the total population size, the characteristic dispersal rate of the population has a drastic influence on the geometry of optimal resource distributions (in the sense that, the lower the characteristic dispersal rate, the more spread out the optimal resource distribution). Here, we provide an \emph{explicit blow-up rates} for the $BV$-norm of optimal resource distribution, the $BV$-norm being a natural way to quantify the fragmentation or complexity of a resource distribution. {We refer to Section~\ref{Se:Model} for further explanations.}
\end{itemize}

%{*** j'ai enlev\'e le plan de l'introduction. \c Ca me m'a pas sembl\'e naturel, on écrit plut\^ot le plan de l'article complet en g\'en\'eral. *** }

%This introduction is structured accordingly, and has the following plan
%\begin{itemize}
%\item In Subsection \ref{Se:Model}, we lay down the models and the optimisation problems under consideration. We provide information about the total population size optimisation problem: we motivate the main qualitative properties (bang-bang property, \emph{concentration} or \emph{fragmentation} of resources) we are interested in from a bio-mathematical perspective, and we state informal versions of our first results (namely, Theorems \ref{Th:BangBang}, \ref{Th:Jcroi} and \ref{Th:FragND}).
%\item In Subsection \ref{Se:Elliptic}, we present all our results.
%\item In Subsection \ref{Se:Bib}, we include further discussion of bibliographical references that deal with the optimisation of the total population size or with the influence of diffusivity on this problem.
%\end{itemize}

\subsection{Model and statement of the problems}\label{Se:Model}
%\subsubsection{Optimisation of the total population size (and other criteria) in the elliptic case}
\paragraph{Statement of the problems}
Let us first lay down the model and the optimisation problems under consideration. The following paragraph is dedicated to explaining which kind of properties we want to obtain for these optimisation problems.

%In the present article, we present what we believe to be, up to this date, the most precise description of solutions of the problem of optimising the total population size for the single-species logistic-diffusive equation. Let us first 
%{Ici, j'ai comment\'e une phrase disant que cet article va plus loin que les autres sur le sujet. Comme on est en train d'\'enoncer les pbs, cela ne me semble pas \^etre le bon endroit pour dire cela.}

We introduce the model we consider throughout the paper. We place ourselves in the framework of the Fisher-KPP equation which, since the seminal \cite{Fisher,KPP}, has been used at length: while its apparent simplicity makes it amenable to mathematical analysis, it is complex enough to capture several fundamental aspects of population dynamics \cite{Skellam}. This model reads:
\begin{equation}\label{LDE}\tag{$\bold E_{m,\mu}$}%\label{Eq:E0}
\left\{\begin{array}{ll}
\mu \Delta \theta+\theta(m-\theta)=0& \text{ in }\O, 
\\ \frac{\partial \theta}{\partial \nu}=0& \text{ on }\partial \O, 
\\ \theta\geq 0, \theta\neq 0,
\end{array}\right.
\end{equation}
where $\theta:\O\to \R_+$ is the population density,.
The population access to resources is modelled by a function $m\in L^\infty(\O)$ and $\mu>0$ is the dispersal rate. 

Although we consider here Neumann boundary conditions, Theorems \ref{Th:BangBang} and \ref{Th:Jcroi} below {can be extended to Robin boundary conditions as well}, the only difficulty being that one would need to ensure existence and uniqueness for the logistic-diffusive equations under these conditions. We comment on this in the conclusion (Section~\ref{Se:BC}).

Provided that $m\geq 0$ and $m\not\equiv 0$, {there exists a unique solution} to \eqref{LDE} \cite{BHR,CantrellCosner1,MR1105497}. We {denote} it $\theta_{m,\mu}$.
%{*** Ici, on r\'e\'ecrivait KPP en changeant $\theta$ en $\theta_{m,\mu}$, alors qu'on vient de le faire au-dessus. J'ai comment\'e cette partie.***}
% and the main equation reads:
%\begin{equation}\label{LDE}\tag{$\bold E_{m,\mu}$}
%\begin{cases}
%\mu \Delta \theta_{m,\mu}+\theta_{m,\mu}(m-\theta_{m,\mu})=0\text{ in }\O, 
%\\ \frac{\partial \tmm}{\partial \nu}=0\text{ on }\partial \O, 
%\\ \theta_{m,\mu}\geq 0, \theta_{m,\mu}\neq 0.\end{cases}\end{equation} 

We can hence define the total-population size functional 
\begin{equation}\label{def:Fmu}
\forall \mu>0,\forall m\in \mathcal M(\O),\quad   F_\mu(m):=\int_\O \theta_{m,\mu}.
\end{equation}  
We use the following class of constraints on the admissible resource distributions $m$, which was introduced in \cite{Lou2008} and used, for instance, in \cite{Mazari2020,NagaharaYanagida}:
\begin{equation}\label{Eq:Ad}\mathcal M(\O):=\left\{m\in L^\infty(\O), 0\leq m\leq 1, \int_\O m=m_0\right\}.\end{equation} 
The parameter $m_0$ is a positive real number such that $m_0<|\O|$, {where $|\O|$ denotes the volume of $\O$}, in order to ensure that $\mathcal M(\O)\neq \emptyset$. %This admissible class is also the one commonly used to handle other optimisation problems in population dynamics, such as the optimal survival problem; we describe this problem in the following paragraph. 
{The $L^1$ constraint accounts for the fact that, in a given domain, only a limited amount of resources is available. The second constraint is a pointwise one, and accounts for natural limitations of the environment, i.e. the fact that, in a single spot, only a maximum amount of resources can be available. }

%In essence, the $L^\infty$ and $L^1$ constraints encode the minimal information necessary to ensure existence of a solution to the optimisation problem. %Regarding the value of the lower and upper bounds on $m$, we briefly comment on them in Remark \ref{Re:Autres} below.

{The optimisation problem we consider} reads

\begin{equation}\tag{$P_\mu$}\label{Eq:Pv}
\boxed{
\sup_{m\in \mathcal M(\O)}F_\mu(m)
},
\end{equation}
{where $F_\mu(m)$ is given by \eqref{def:Fmu}.}

\begin{remark}[Existence of maximisers] For any $\mu>0$, the existence of a solution $m_\mu^*$ of \eqref{Eq:Pv} is an immediate consequence of the direct method in the calculus of variations.\end{remark}

%{*** Ici, il y avait une rque sur une g\'en\'eralisation du thrm qu'on n'a pas encore \'enonc\'e. Je l'ai d\'eplac\'ee apr\`es l'\'enonc\'e du thrm principal.***}

In the following paragraph, we present the fundamental properties we are interested in.

\paragraph{Optimisation of spatial heterogeneity in mathematical biology: fundamental properties under consideration}\label{Se:Scope}

 Starting from spatially homogeneous models \cite{Fisher,KPP}, in which a population is assumed to live in a homogeneous environment, mathematical biology has over the past decades started considering the impact of \emph{spatial heterogeneity} on population dynamics \cite{MR1105497}. In most works, this spatial heterogeneity is modelled using resource distributions. Mathematically, this amounts to taking into account the heterogeneity in the reaction term of the equation. Given that it is hopeless, for a given resource distribution, to attain an explicit description of the ensuing population dynamics, the focus has, more recently, shifted to an optimisation point of view. 
 
{This approach has been initiated in} \cite{BHR,KaoLouYanagida,LouInfluence} and has since received a considerable amount of attention \cite{BaiHeLi,Ding2010,LamboleyLaurainNadinPrivat,LouNagaharaYanagida,MazariNadinPrivat,Mazari2020,NagaharaYanagida}. The initial question that motivated most of these works was related to the optimal survival ability of a population \cite{BHR,ShigesadaKawaski}. Namely:
\begin{center}
\emph{What is the best way to spread resources in a domain to ensure the optimal survival of a population?}\end{center} This problem is by now very well understood in several simple cases (we provide ampler references in Section~\ref{Se:Bib}). {Among all the issues tackled by the authors of} \cite{BHR,KaoLouYanagida,LamboleyLaurainNadinPrivat}, let us single out the following ones, which have been deemed crucial in the study of spatial heterogeneity as they provide simple, qualitative information about the influence of heterogeneity: in a domain $\O$, if we consider resource distribution $m$ belonging to $\mathcal M(\O)$ defined by \eqref{Eq:Ad},
%has to satisfy some pointwise ($L^\infty$) bounds and some global ($L^1$) bound of the form  $0\leq m\leq 1$ (i.e. $m\in \mathcal M(\O)$ defined in \eqref{Eq:Ad}), then
\begin{enumerate}
\item {\textsf{does the bang-bang property hold at the optimum?} In other words, if one looks at maximising  a criterion over resource terms in $\mathcal M(\O)$, %and under these $L^\infty$ and $L^1$ constraints, 
does any optimal resource distribution $m^*$ write $m^*=\mathds 1_E$, for some measurable subset $E$ of $\O$ of positive measure?} Alternatively, this means that the underlying domain $\O$ can be decomposed as
\begin{equation}\O=\{m^*=1\}\sqcup \{m=0\}.\end{equation} 

%Henceforth, we say that $m$ is a bang-bang function if $m=\mathds 1_E$ for some measurable subset $E$.

Despite several partial results \cite{Mazari2020,NagaharaYanagida} which we detail in Remarks \ref{Re:CompaNY} and \ref{Re:CompaMNP} , this property is not known to hold in general for the optimisation of the total population size. In this article we prove that this property indeed holds for the optimal population size whatever the value of $\mu>0$ be (Theorem \ref{Th:BangBang}). 
%Our result reads:
%\begin{paragraph}{Theorem.}
%\emph{For any diffusivity $\mu>0$, any solution $m_\mu^*$ of \eqref{Eq:Pv} is a bang-bang function: there exists $E_\mu^*$ such that $m_\mu^*=\mathds 1_{E_\mu^*}$.}
%\end{paragraph} 
%In addition to proving this {result}, we give an alternative proof of an earlier result of Nagahara and Yanagida \cite{NagaharaYanagida} which relies on a regularity assumption on maximisers (see Section~\ref{Se:Elliptic}, Remark~\ref{Re:CompaNY} for a precise statement). It is notable that our result does not require such assumptions.
%{*** Ici, il y avait une rque sur NY, que j'ai enlev\'ee, car je ne la trouvais pas assez pr\'ecise et aussi car on en parle plus en d\'etails par la suite. J'ai enlev\'e l'\'enonc\'e du thm car il s'agissait du m\^eme que celui donn\'e dans la section suivante. ***}

%The same conclusion holds for solutions of \eqref{Eq:PvJ}, see Theorem \ref{Th:Jcroi}.
\item {\textsf{do optimal resources tend to concentrate?}} In \textquotedblleft simple \textquotedblright cases (i.e. in simple geometries and for specific boundary conditions), optimal resource distributions for the survival ability \cite{BHR,KaoLouYanagida} are known to be \emph{concentrated}. {For instance, considering an optimal resource distribution for the survival ability, {which is known to write $m^* =\mathds 1_E$}, then the set $E$ is connected and enjoys moreover a symmetry property %{*** Ici, on disait d\'ecroissant dans toutes les directions, mais le mot direction pr\^etait \`a confusion, j'ai modifi\'e ***}
%$m^*$ is decreasing in every direction
 for Neumann boundary conditions in an orthotope \cite[Proposition 2.9]{BHR}. A similar conclusion holds whenever $\O=\mathbb B(0;r)$ is a ball and if Dirichlet boundary conditions are imposed rather than Neumann. In that case, $E=\mathbb B(0;r^*)$ is another centered ball, with a radius $r^*$ chosen so as to satisfy the volume constraint.}  For general geometries and Robin boundary conditions, the situation is very involved and we refer to \cite{LamboleyLaurainNadinPrivat} for up to date qualitative properties. Such results are a mathematical formalisation of a paradigm first stated in \cite{ShigesadaKawaski}: fragmenting the set $\{m=1\}$ leaves less chance for survival: \emph{concentrating resources is favorable to population dynamics}. 

%We give a graphical representation of this concentration (or, conversely, fragmentation) in Figure \ref{Fig1}.
%
%
%\begin{figure}[H]
%\begin{center}
%\includegraphics[width=5.2cm]{concentration.png}
%\includegraphics[width=5.2cm]{fragmentation.png}
%\caption{Concentration and fragmentation of resources}
%\label{Fig1}
%\end{center}
%\end{figure}

In the case of the total population size, it was first noticed in \cite{Mazari2020} that such results do not in general hold for small diffusivities, where the geometry of the optimal resource distribution tends to become more complicated. Recently, in \cite{LouNagaharaYanagida}, a complete treatment of a spatially discretised version of the problem was carried out, and precise fragmentation rules were established. However, these results cannot be extended to the present continuous version, since the optimiser they compute strongly depends on the discretization scale.  In \cite{MRBSIAP}, it was shown that, the slower the dispersal rate of the population, the bigger the $BV$-norm\footnote{Recall that the total variation semi-norm of a function is 
\begin{equation}\left| m\right|_{TV(\O)}=\sup\left\{\int_\O m \operatorname{div}(\p), \p \in \mathscr C^1_c(\O;\R^d), \Vert \p \Vert_{L^\infty}\leq 1\right\}\end{equation} and that the bounded variation norm of $m$ is in turn defined as
\begin{equation}
\Vert m\Vert_{BV(\O)} = \Vert m\Vert_{L^1(\O)}+\left| m\right|_{TV(\O)}.
\end{equation}
} of the optimal resource distribution {is}.
%, where the $BV$-norm of a function is the bounded variation norm of the function. 

\begin{remark}\label{Re:BV} When $m\in W^{1,1}(\O)$, the $BV$-norm and the $W^{1,1}$ norm coincide. When $m=\mathds 1_E$ and $m$ is a Cacciopoli set (i.e. a set with finite Cacciopoli perimeter) then $\Vert m\Vert_{BV(\O)}=|E|+\operatorname{Per}(E)$, where $\operatorname{Per}(E)$ is the Cacciopoli perimeter of the set. As a consequence, in our context, an information on the blow-up rate of the $BV$-norm {yields} an information on the blow-up rate of the $TV$-norm and, since Theorem \ref{Th:BangBang} ensures that any maximiser $m_\mu^*$ writes as $\mathds 1_{E_\mu^*}$, this implies a blow-up rate on $\operatorname{Per}(E_\mu^*)$ as $\mu \to0^+$. We refer to \cite{AmbrosioFuscoPallara} for more information regarding functions of bounded variations and perimeters of sets.
\end{remark}

{In \cite{MRBSIAP}, the main result reads: 
\begin{paragraph}{Theorem \textnormal{\cite[\textit{Theorem 1}]{MRBSIAP}}.}
\emph{Let $\O=(0;1)^d$, $\mu>0$, and let  $m_\mu^*$ denote a solution of Problem~\eqref{Eq:Pv}. Then, 
$$\Vert m_\mu^*\Vert_{BV(\O)}\xrightarrow[\mu \to0^+]{} +\infty.$$}
\end{paragraph}
 In this article, we quantify this result by explicitly identifying blow-up rates in terms of the characteristic dispersal rate, and provide a scaling we expect to be optimal (Theorem~\ref{Th:FragND}). 
 %We give a proof in the one-dimensional case and another one in the multi-dimensional case. 
%\begin{center}\emph{If $\O=\prod_{i=1}^d (0;1)$, there exists a constant $C_0>0$ such that, if $\mu>0$ is small enough and $m_\mu^*$ is a solution of \eqref{Eq:Pv}, then 
%$$\Vert m_\mu^*\Vert_{BV(\O)}\geq \frac{C_0}{\sqrt{\mu}}.$$}\end{center}
%It is notable that this scaling does not depend on the dimension (although the constant $C_0$ does).
%A more complete version of this result is provided in Theorem~\ref{Th:FragND}. 
The proof relies on fine energy estimates. }
\end{enumerate}

A more in-depth discussion of the bibliography is included in Section \ref{Se:Bib}.% of this Introduction.

\subsection{Main results}\label{Se:Elliptic}
\subsubsection{The bang-bang property}
{Let us first state that every solution of the optimal population size problem is {\it bang-bang}. This property, intrinsically interesting, has a practical interest: it allows us to reformulate the problem as a shape optimization one, the unknown being the set in which $m$ takes its maximum value. One can then use adapted numerical approaches. }

\begin{theorem}\label{Th:BangBang} Let $\O\subset \R^d$ be a bounded {connected} domain {with a $\mathscr{C}^2$ boundary}. Let $m_\mu^*$ be a solution of \eqref{Eq:Pv}. Then 
%$m_\mu^*$ is a bang-bang function: 
there exists a measurable subset $E\subset \O$ such that 
\begin{equation}m_\mu^*=\mathds 1_E.\end{equation}\end{theorem}

\begin{remark}[{Sketch of the proof}]
The idea of the proof {rests upon the following fact}: 
%\emph{via} elementary computations, 
{
we can actually show that the second order G\^ateaux derivative of the criterion $F_\mu$ at a point $m\in \mathcal M(\O)$ in a direction $h$ (such that $m+th\in \mathcal M(\O)$ for $t$ small enough) writes
% in the form
\begin{equation}
\ddot F_\mu(m)[h,h]=\int_\O \Psi_m(x)|\n \dtmm|^2-\int_\O \Phi_m(x)\dtmm^2,
\end{equation} 
where  $\Psi_m, \Phi_m\in L^\infty(\O)$, $\inf_\O \Psi_m>0$ and $\dtmm$ solves a PDE of the kind
$$
\left\{\begin{array}{ll}
\mathcal L_m\dtmm
%\left(:=-\Delta \dtmm-\dtmm(m-2\tmm)\right)
=h\tmm & \text{in }\O\\
\partial_\nu \dtmm=0 & \text{in }\partial\O,
\end{array}\right.$$
where $\mathcal L_m$ denotes an elliptic operator of second order.
%, with homogeneous Neumann boundary conditions.
 We then argue by contradiction, assuming the existence of a maximiser $m_\mu^*$ that is not a bang-bang function, meaning that the set $\{0<m_\mu^*<1\}$ is of positive Lebesgue measure. Using the expression of $\ddot F_\mu(m)[h,h]$ above, we exhibit a function $h$ in $L^\infty$ supported in $\{0<m_\mu^*<1\}$, with $\int_\O h=0$, such that $\int_\O |\n \dot\theta_{m_\mu^*,\mu}|^2$ is much larger than $\int_\O \dot\theta_{m_\mu^*,\mu}^2$. 
 This is done by using the Fourier (spectral) expansion of $\tmm$, associated with the operator $\mathcal L_{m_\mu^*}$, and by choosing $h$ as above, and such that, $h\theta_{m_\mu^*,\mu}$ only has high Fourier modes in this basis.}
\end{remark}
\begin{remark}[Comparison with the results of \cite{NagaharaYanagida}] \label{Re:CompaNY}In \cite{NagaharaYanagida}, the following result is {proved}: if $m\in \mathcal M(\O)$ is such that $\{0<m<1\}$ has a non-empty interior, then it is not a solution of \eqref{Eq:Pv}. This in particular implies that, if a maximiser $m_\mu^*$ of the total population size functional is Riemann integrable, then $m_\mu^*$ {is continuous almost everywhere in $\O$} and is thus necessarily of bang-bang type. However, such regularity is usually extremely hard to prove, and it is unclear to us whether it is attainable in this context. Furthermore, we believe we have located a slight mistake in their proof and we hence provide a correction {in Section~\ref{Se:Comparison}}, where we also comment on the comparison between our two proofs.
\end{remark}
\begin{remark}
%[Comparison with the results of \cite{Mazari2020}]
\label{Re:CompaMNP} In \cite{Mazari2020}, the bang-bang property is proved to hold whenever the diffusivity $\mu$ is large enough{, using a proof that is also based on a second derivative calculation, but whose philosophy is completely different from that of Theorem~\ref{Th:BangBang}}. Our present result does not require such an assumption.  
\end{remark}

\begin{remark}
A minor adaptation of our proof allows us to handle more general admissible sets and criteria: 
\begin{itemize}
\item let us consider a function $j$ satisfying 
\begin{equation}\label{Hyp:J}\tag{$\bold H_j$}
j\in \mathscr C^{2}([0;1];\R), j\text{ is increasing in $[0;1]$: } j'>0.\end{equation}
We define, for any $\mu>0$, 
\begin{equation}\mathcal J_{j,\mu}:\mathcal M(\O)\ni m\mapsto \int_\O j(\tmm)\end{equation} and the optimisation problem 
\begin{equation}\label{Eq:PvJ}\tag{$P_{j,\mu}$}\fbox{$\displaystyle \sup_{m\in \mathcal M(\O)}\mathcal J_{j,\mu}(m).$}\end{equation} Then proving a bang-bang property for this problem is amenable to analysis using our technique.
%, if the goal is to prove a bang-bang type property (see the next paragraph for the presentation of the bang-bang property).
\item  If one were to change the $L^\infty$ bounds on $m$ to $0\leq m\leq \kappa$ for some positive $\kappa$, the only modification would be to replace $[0;1]$ with the interval $[0;\kappa]$ in assumption \eqref{Hyp:J} above. 
\end{itemize}
{We claim that our method of proof fits immediately to show the following result: }
\begin{theorem}\label{Th:Jcroi}
 Let $\O\subset \R^d$ be a $\mathscr C^2$ bounded domain and let $j$ satisfying \eqref{Hyp:J}. Let $m_{\mu,j}^*$ be a solution of \eqref{Eq:PvJ}. Then $m_{\mu,j}^*$ is a bang-bang function: there exists a measurable subset $E\subset \O$ such that 
\begin{equation}m_{\mu,j}^*=\mathds 1_E.
\end{equation}
\end{theorem}
{A short paragraph explaining how to adapt the proof of Theorem~\ref{Th:BangBang} is provided in Section~\ref{secproofTh:Jcroi}.} 
\end{remark}

\subsubsection{Quantifying the fragmentation for small diffusivities}
Our second main result deals with the aforementioned fragmentation property for low diffusivities. 
Here, we will be led to make stronger assumptions on $\O$, namely, that 
$\O$ is an orthotope: $\O=(0;1)^d$. {Hence, according to \cite[Lemma 2]{MRBSIAP}, one has} 
\begin{equation}\label{Eq:Arvo}\underset{\mu\to 0^+}{\lim\inf}\left(\sup_{m\in \mathcal M(\O)}F_\mu(m)\right)>m_0=\inf_{\mu>0, m\in \mathcal M(\O)}F_\mu(m).\end{equation} 
The equality on the right-hand side is obtained in \cite[Theorem 1.2]{LouInfluence}.\begin{remark} {If we consider  another domain $\tilde \O$ such that \eqref{Eq:Arvo} holds, then the main result below, Theorem \ref{Th:FragND}, holds in $\tilde \O$.}\end{remark}

We provide hereafter an explicit blow-up rate that we believe to be optimal. {Once again, let us emphasize that it this rate does not depend on the space dimension $d$. }

\begin{theorem}\label{Th:FragND}
Let $d\geq 1$ and let $\O=(0;1)^d$. There exists $C_0>0$ such that  the following holds: there exists $\mu_0>0$ such that, for any $\mu\in (0;\mu_0)$, if $m_\mu^*$ is a solution of \eqref{Eq:Pv}, then
\begin{equation}
\Vert m_\mu^*\Vert_{BV(\O)}\geq \frac{C_0}{\sqrt{\mu}}.
\end{equation}
%Alternatively:\begin{equation}\underset{\mu\to 0^+}{\lim\inf}\left(\sqrt{\mu}\Vert m_\mu^*\Vert_{BV(\O)}\right)>0.\end{equation}
%{*** j'ai enlev\'e le ``alternatively blabla'' avec la liminf. Je ne voyais pas trop ce que cela ajoutait.***}
\end{theorem}

\begin{remark}[Comment on the proof of Theorem \ref{Th:FragND}] The crux of the proof is the variational formulation of \eqref{LDE}, which ensures that $\tmm$ is the unique minimiser of 
\begin{equation}
\mathcal E_{m,\mu}:\left\{u\in W^{1,2}(\O), \ u\geq 0\right\}\ni u\mapsto \frac{\mu}2\int_\O |\n u|^2-\frac12\int_\O mu^2+\frac13\int_\O u^3,
\end{equation} and which needs to be carefully estimated as $\mu \to 0^+$. 
%We could not locate in the literature and we give a proof of it in Lemma \ref{Le:VF}. 
We prove that a "shifted" version of this energy controls the quantity $\Vert \tmm-m\Vert_{L^1(\O)}$ (Lemma~\ref{Le:L1}). Therefore, using estimate~\eqref{Eq:Arvo}, we aim at controlling $\mathcal E_{m,\mu}(\tmm)$ as $\mu \to 0^+$. 
Using Modica-type estimates, {one can show} that, for a fixed $m\in \mathcal M(\O)$ that writes $m=\mathds 1_E$, there holds 
$$
\sqrt{\mu}\mathcal E_{m,\mu}(\tmm)\xrightarrow[\mu\to 0^+]{} \operatorname{Per}(E).
$$ 
However, this convergence is non-uniform {with respect to} $m$ (or, more precisely to $E$) and, since we are working with a maximisation problem, it is not possible to conclude using the convergence result above. In the one dimensional case, we propose, in the appendix, an adaptation of \cite{ModicaMortola} that makes this strategy work nonetheless.  In higher dimension, we estimate the energy using a regularisation of $m$ as a test function in the energy formulation of the equation.

\end{remark}

\subsection{Bibliographical comments on \eqref{Eq:Pv}}\label{Se:Bib}
In this section, we gather a discussion {on references connected to }the optimisation of the total population size in logistic-diffusive models. For a presentation of the literature devoted to the optimal survival ability, we refer to \cite[Introduction]{MazariThese}. 

\paragraph{Influence of the diffusivity $\mu$ on $F_\mu$.} 
%We start by describing several works that contain results on the influence of the parameter $\mu>0$ on \eqref{Eq:Pv}.

Problem~\eqref{Eq:Pv} {has been first introduced in} \cite{Lou2008} and several properties had been derived in \cite{LouInfluence}, one of which is the following: for {every} $\mu>0$, the unique minimiser of $F_\mu$ in $\mathcal M(\O)$ is $m_0$; in other words
\begin{equation}
\forall \mu>0, \ \forall m\in \mathcal M(\O), \quad m(\cdot)\neq m_0\Rightarrow  F_\mu(m)>m_0.\end{equation} {This result means that} spatial homogeneity is detrimental to the population size. Furthermore, it is proved in \cite{LouInfluence} that, {if $m\in \mathcal M(\O)$ is given}, then 
\begin{equation}\label{Eq:Lou}
{F_\mu(m)\xrightarrow[\mu \to 0^+]{} m_0, \quad\text{and}\quad F_\mu(m)\xrightarrow[\mu \to \infty]{} m_0.}
\end{equation} 
Hence, for a given $m\in \mathcal M(\O)$, the low and high diffusivity limits of the functional correspond to global minima. However, it was proved in \cite[Lemma 2]{MRBSIAP} that 
$$\underset{\mu\to 0^+}{\lim\inf}\left(\sup_{m\in \mathcal M(\O)}F_\mu(m)\right)>m_0,$$ showing the intrinsic difficulty of passing to the low-diffusivity limit in problem~\eqref{Eq:Pv}. 

This point of view, where the resource distribution is considered fixed and the diffusivity is taken as a variable, was later deeply analysed in several articles. Notable among these are the following results:
\begin{enumerate}
\item In \cite{BaiHeLi}, {for a fixed $m\in L^\infty(\O)$ such that $m(\cdot)\geq 0$ and $m(\cdot)\neq 0$}, the authors consider the optimisation problem
\begin{equation}\label{Eq:BHL}\sup_{\mu>0}\left(E_\mu(m):=\frac{F_\mu(m)}{\int_\O m}\right)\end{equation} and observe that, {in the one-dimensional case $\O=(0;1)$}, there holds 
\begin{equation}E_\mu(m)\leq 3.\end{equation} This bound is sharp (a maximising sequence is explicitly constructed) and is not reached by any function $m$. This work has been later extended {to} the higher-dimensional case in \cite{InoueKuto} and the authors prove that, in that case (i.e. in dimension $d\geq 2$), there holds
\begin{equation}
\sup_{\substack{m\in L^\infty(\O)\\ m\geq 0, \ m\neq 0}}\sup_{\mu>0}\, E_\mu(m)=+\infty.
\end{equation}
\item In \cite{LiangLou}, a function $m$ such that the map $\mu \mapsto F_\mu(m)$ has several local maxima is constructed. {It emphasizes} the intrinsic complexity of the interplay between the population size functional and the parameter $\mu>0$.
\end{enumerate}

Finally, let us also note that a related problem, where the underlying model is a system of ODEs with identical migration rates, was considered in \cite{LiangZhang}.
 
We also point out to two surveys \cite{LamLiuLou,MNPChapter} and to the references therein for up-to-date considerations about the influence of spatial heterogeneity for single or multiple species models or for optimisation problems in mathematical biology.

\section{Proofs of Theorems \ref{Th:BangBang} and \ref{Th:Jcroi}}
\subsection{Proof of Theorem \ref{Th:BangBang}}

%\begin{proof}[Proof of Theorem \ref{Th:BangBang}]
The proof of this Theorem relies on a new formulation of the second order optimality conditions for the problem \eqref{Eq:Pv}. {Let us first compute the necessary optimality conditions of the first and second orders. }
\paragraph{Computation of optimality conditions}

It is established in \cite[Lemma 4.1]{Ding2010} that, for any $\mu>0$ the map $\mathcal M(\O)\ni m\mapsto \tmm$ is differentiable {at the first order} in the {sense of }G\^ateaux. {Adapting their proof yields without difficulty its second order G\^ateaux-differentiability}. Let us fix  $m\in \mathcal M(\O)$ and  an admissible perturbation\footnote{{The wording ``admissible perturbation'' means that $h$ belongs to the tangent cone to the set $\mathcal{M}(\O)$ at $m$.
%, denoted by $\mathcal{T}_{m,\mathcal{M}(\O)}$. 
It corresponds to the set of functions $h\in L^\infty(\O)$ such that, for any sequence of positive real numbers $\varepsilon_n$ decreasing to $0$, there exists a sequence of functions $h_n\in L^\infty(\O)$ converging to $h$ as $n\rightarrow +\infty$, and $m+\varepsilon_nh_n\in\mathcal{M}(\O)$ for every $n\in\N$.}\label{footnote:cone}
} $h\in L^\infty(\O)$. 
%such that, whenever $t\in(0;1)$, $m+th\in \mathcal M(\O)$. 
Let us {denote by} $\dtmm$ (resp. $\ddtmm$) the first (resp. second) G\^ateaux-derivative of $\theta_{\cdot,\mu}$ at $m$ in the direction $h$. It is {standard} (we refer to \cite[Lemma 4.1]{Ding2010}) to see that $\dtmm$ solves
\begin{equation}\label{LDEdot}
\left\{\begin{array}{ll}
\mu \Delta \dot\theta_{m,\mu}+(m-2\tmm) \dot\theta_{m,\mu}=-h\tmm & \text{ in }\O,\\
\frac{\partial  \dot\theta_{m,\mu}}{\partial \nu}=0 & \text{ on }\partial \O.
\end{array}\right.
\end{equation} 
\begin{remark}The fact that $\dtmm$ is uniquely determined by that equation (in other words, that \eqref{LDEdot} has a unique solution can be proved as in \cite{Ding2010,Mazari2020}. For the sensitivity analysis and computation of the G\^ateaux-derivatives, we also refer to \cite{NagaharaYanagida}.\end{remark} 
To derive a tractable equation for the G\^ateaux derivative $\dot{F_\mu}(m)[h]$ of the functional $F_\mu$ at $m$ in the direction $h$, let us introduce the adjoint state $\pmm$ as the solution of
\begin{equation}\label{Eq:Ajoint}
\left\{\begin{array}{ll}
\mu \Delta \pmm +\pmm(m-2\tmm)=-1 & \text{in }\O,\\ 
\frac{\partial \pmm}{\partial \nu}=0 & \text{on }\partial\O,
\end{array}\right.
\end{equation} so that, multiplying \eqref{LDEdot} by $\pmm$ and integrating by parts readily gives
\begin{equation}
\int_\O \pmm \tmm h=\int_\O \dtmm=\dot F_\mu(m)[h].
\end{equation}
{It is standard in optimal control theory (see e.g. \cite{MR1155489})} that, if $m_\mu^*$ is a solution of \eqref{Eq:Pv} then there exists a constant $c$ such that 
\begin{equation}\label{Eq:Opt1}\{0<m_\mu^*<1\}\subset \{\theta_{m_\mu^*}p_{m_\mu^*,\mu}=c\}.\end{equation} 

\begin{remark}As is done in \cite{Mazari2020}, the sets $\{m_\mu^*=1\}$ and $\{m_\mu^*=0\}$ can be described in terms of level sets of the so-called switching function $\tmm\pmm$ but we do not detail it since these are not informations we will use in the proof.
\end{remark}

Let us turn to the computation of the second order G\^ateaux  derivative of the functional $F_\mu$ in the direction $h$, which will be denoted $\ddot F_\mu(m)[h,h]$. To obtain it, we first recall  (see \cite[Equation (18)]{Mazari2020}) that $\ddtmm$ solves

\begin{equation}\label{LDEdotdot}
\left\{\begin{array}{ll}
\mu \Delta \ddot\theta_{m,\mu}+(m-2\tmm) \ddot\theta_{m,\mu}=-2h\dot\theta_{m,\mu}+2\dot\theta_{m,\mu}^{2} & \text{ in }\O,\\
\frac{\partial  \ddot\theta_{m,\mu}}{\partial \nu}=0 & \text{ on }\partial \O.
\end{array}\right.
\end{equation}

Multiplying \eqref{LDEdotdot} by $\pmm$ and integrating by parts yields

\begin{align*}
\int_{\O}\ddot\theta_{m,\mu} &= 2\int_{\O}\left(h\dot\theta_{m,\mu}-\dot\theta_{m,\mu}^{2}\right)\pmm =2\int_{\O}\left(\displaystyle\frac{-\mu \Delta \dot\theta_{m,\mu}-(m-2\tmm) \dot\theta_{m,\mu}}{\tmm}\dot\theta_{m,\mu}-\dot\theta_{m,\mu}^{2}\right)\pmm\\
&=2\int_{\O}\left(-\mu \Delta \dot\theta_{m,\mu}-(m-\tmm) \dot\theta_{m,\mu}\right) \frac{p \dot\theta_{m,\mu}}{\tmm}.\end{align*}
Let us {introduce} $\umm:=\frac{\pmm}\tmm$. We thus obtain
\begin{eqnarray}
\int_{\O}\ddot\theta_{m,\mu} 
&=& 2\int_{\O}\left(\mu \nabla (\umm \dot\theta_{m,\mu}) \nabla\dot\theta_{m,\mu}-(m-\tmm) \dot\theta_{m,\mu}^{2}\umm\right) \nonumber \\
&=&2\int_{\O}\umm\left(\mu|\nabla\dot\theta_{m,\mu}|^{2}-\left(m-\tmm+\frac{\Delta \umm}{2\umm}\right) \dot\theta_{m,\mu}^{2}\right).\label{expr:thetaddot}
\end{eqnarray}

Furthermore, it is straightforward to see that
\begin{equation}\label{Eq:InfTheta}
\forall m\in \mathcal M(\O),\qquad  \inf_\O \tmm>0.
\end{equation} 
Furthermore, we have the following result:
\begin{lemma}\label{Cl:InfP}
For every $m\in \mathcal M(\O)$, \begin{equation}\inf_\O \pmm>0.\end{equation}
\end{lemma}
\begin{proof}[Proof of Lemma~\ref{Cl:InfP}]
We start from the observation that $\tmm$ solves \eqref{LDE} implies that the principal eigenvalue $\lambda(m-\tmm,\mu)$ of the operator $-\mu\Delta-(m-\theta)\operatorname{Id}$ is zero \cite{LouInfluence}. Since $\tmm>0$ in $\O$, the first eigenvalue $\lambda(m-2\tmm,\mu)$ of the operator $\mathcal L_m:=-\mu\Delta-(m-2\tmm)\operatorname{Id}$ satisfies
\begin{equation}\label{Eq:Halm}
\lambda(m-2\tmm,\mu)>0,
\end{equation}
 as a consequence of the monotonicity of eigenvalues \cite{DockeryHutsonMischaikowPernarowskiEvolution}.
Since $\pmm$ satisfies $\mathcal L_m\pmm=1>0$ with Neumann boundary conditions, the conclusion follows from multiplying the equation on $\pmm$ by the negative part $\left(\pmm\right)_-$ and integrating by parts: it yields
\begin{equation}
\mu\int_\O |\n (\pmm)_-|^2-\int_\O (\pmm)_-^2(m-2\tmm)=-\int_\O (\pmm)_-<0\text{ if }(\pmm)_-\neq 0.
\end{equation} 
However, {according to the Courant-Fischer principle},
\begin{equation}
\lambda(m-2\tmm,\mu)=\inf_{\substack{u\in W^{1,2}(\O)\\ \int_\O u^2=1}}\mu\int_\O |\n u|^2-\int_\O u^2(m-2\tmm)>0
\end{equation} 
and therefore, it follows that $\pmm(\cdot)\geq 0$ and $\pmm(\cdot)\neq 0$ in $\O$. To {conclude}, it suffices to apply the strong maximum principle.
\end{proof}

{According to Lemma~\ref{Cl:InfP} and \eqref{Eq:InfTheta}}, it follows that $\umm$ satisfies 
\begin{equation}\label{Eq:U} 
\inf_{\O}\umm>0.
\end{equation} 
Furthermore, standard elliptic estimates entail 
\begin{equation}
\forall p\in (1;+\infty),\quad  \tmm, \pmm \in W^{2,p}(\O),
\end{equation}
and from Sobolev embeddings, we get 
\begin{equation}\label{Eq:BorneC1}
\tmm,\pmm\in \mathscr C^{1,\alpha}({\overline{\O}})
\end{equation}  for any $\alpha \in (0;1)$. {Using the equations on $\tmm$ and $\pmm$, this gives, in turn that $\Delta \tmm$ and $\Delta \pmm$ belong to $L^\infty(\O)$.
It follows, by computing explicitly $\Delta \umm$, that $\Delta \umm$ belongs to $L^\infty(\O)$.}

 If we then define
$V_{m,\mu}:=\left(m-\tmm+\frac{\Delta \umm}{2\umm}\right)$ we have, as a consequence, that 
\begin{equation}\label{Eq:V} 
V_{m,\mu}\in L^\infty(\O).
\end{equation}
{Starting from \eqref{expr:thetaddot},} $\ddot F_\mu(m)[h,h]$ {rewrites} in the more tractable form
\begin{equation}\label{Eq:DSD}\ddot F_\mu(m)[h,h]=\int_\O \ddtmm=\mu \int_\O \umm |\n \dtmm|^2-\int_\O V_{m,\mu}\dtmm^2.
\end{equation}

{This expression is crucial to proving Theorem~\ref{Th:BangBang}.}
\begin{proof}[Proof of Theorem \ref{Th:BangBang}]
 Let us argue by contradiction, assuming the existence of a maximiser $m$ (for the sake of readability, we drop the subscript $m_\mu^*$) of $F_\mu$ in $\mathcal M(\O)$ such that {the set
$\tilde \O:=\{0<m<1\}$ is of positive Lebesgue measure.}

Our goal is now to construct {an admissible perturbation $h\in L^\infty(\O)$ (see Footnote~\ref{footnote:cone})} such that
\begin{equation}\label{Eq:Toss}
h\text{ is supported in }\tilde \O, \qquad \ddot F_\mu(m)[h,h]>0.
\end{equation}
Let us first note that from the optimality conditions \eqref{Eq:Opt1}, if $h$ is supported in $\tilde \Omega$ and satisfies $\int_{\tilde \Omega}h=0$, then, for the constant $c$ given in \eqref{Eq:Opt1} we have 
$$\dot F_\mu(m)[h]=\int_\O h \tmm \pmm=c\int_{\tilde \O}h=0.$$
{Hence, if $h$ satisfies \eqref{Eq:Toss}, then a Taylor expansion yields $F_\mu(m+\varepsilon h)-F_\mu(m)=\frac{\varepsilon^2}{2}\ddot F_\mu(m)[h,h]+\operatorname{o}(\varepsilon^2)$, which leads to a contradiction whenever $\varepsilon>0$ is chosen small enough.}
{It is standard to show that any perturbation $h$ supported in $\O$ is admissible if, and only if }$\int_\O h=0$.

\begin{remark}
{To implement the previous construction}, it suffices {in fact to construct $h\in L^2(\O)$ so that }\eqref{Eq:Toss} {is satisfied and $\int_\O h=0$, forgetting that $h$ has to belong to $L^\infty(\O)$}. Indeed, let us assume that such a $h\in L^2(\O)$ exists. Then, we introduce the sequence $h_n:=h\mathds 1_{|h|\leq n}-\int_\O h\mathds 1_{|h|\leq n} \in L^\infty(\O)$, which converges weakly in $L^2(\O)$ to $h$ {as $n\to \infty$}. By elliptic regularity, it entails strong $W^{1,2}$-regularity of $\dot \theta_{m,\mu}[h_n]$ to $\dtmm$ as $n\to \infty$ and thus the convergence of second order derivatives. Choosing $n$ large enough yields the required contradiction.
{In what follows, we will }hence look for a function $h\in L^2(\O)$ satisfying \eqref{Eq:Toss} and $\int_\O h=0$. 
\end{remark}

{According to} \eqref{Eq:DSD}, by using \eqref{Eq:V} and \eqref{Eq:U}, there exist two positive constants $A_1$ and $A_2$ such that
\begin{equation}\ddot F_\mu(m)[h,h]\geq A_1\int_\O |\n \dtmm|^2-A_2\int_\O \dtmm^2.\end{equation}
To obtain a contradiction, it hence suffices to construct a perturbation $h\in L^2(\O)$ with support in $\tilde \O$  satisfying $\int_\O h=0$ and such that
\begin{equation}\label{Eq:Fmot}\int_\O|\n \dtmm|^2>\frac{A_2}{A_1}\int_\O \dtmm^2.\end{equation}
Let us prove that such a perturbation $h$ exists. To {this aim}, let us introduce the operator $\mathcal L$ defined by
\begin{equation}
{\mathcal L_m:H^2(\O)\ni \psi\mapsto -\mu \Delta\psi-(m-2\tmm)\psi\in L^2(\O).}
\end{equation} 
{This operator is self-adjoint and of compact inverse in $L^2(\O)$, as a consequence of the spectral estimate \eqref{Eq:Halm}}. As a consequence, there exists a sequence of eigenvalues
\begin{equation}\lambda_1(\mathcal L_m)<\lambda_2(\mathcal L_m)\leq \dots\leq \lambda_k(\mathcal L_m)\xrightarrow[k\to \infty]{}  +\infty,\end{equation} each of these eigenvalues being associated with a $L^2$-normalised eigenfunction $\psi_k$ {solving}
\begin{equation}\begin{cases}
\mathcal L_m \psi_k=\lambda_k(\mathcal L)\psi_k\text{ in }\O, 
\\ \frac{\partial \psi_k}{\partial \nu}=0\text{ on }\partial \O, 
\\ \int_\O \psi_k^2=1.
\end{cases}
\end{equation}

%We recall that by assumption $|\tilde \O|>0$. 
Let us fix $K\in \N\backslash\{0\}$ { that will be fixed later and } consider the family of linear functionals $\{R_k\}_{k=0, \dots,K}\subset \left(L^2\left(\tilde\O\right)'\right)^K$ defined by
{\begin{equation}
\forall f \in L^2(\tilde\O), \quad R_0(f):=\int_\O \mathds 1_{\tilde\O} f\quad \text{and}\quad R_k(f):=\int_\O \mathds 1_{\tilde\O}\tmm \psi_k f
\end{equation}
for every $k\in \llbracket 1,K\rrbracket $.}

Let us define $E_k:=\ker(R_k)$ for every $k\in \llbracket 0,K\rrbracket $. Observe that each space $E_k$ is of codimension at most 1. In particular,
\begin{equation}
E:=\cap_{k=0}^K E_k\subset L^2(\tilde\O)
\end{equation} 
is of codimension at most $(K+1)$ in $L^2(\tilde\O)$ and is non-empty. 
%Since $|\tilde \O|>0$, $L^2({\tilde\O})$ has infinite dimension. 
Let us hence pick $F_K\in E\backslash\{0\}$ {and} assume by homogeneity, that 
\begin{equation}
\int_\O |F_K\mathds 1_{\tilde\O} \tmm|^2=1.
\end{equation} 
Let us extend $F_K$ to $\O$ by setting {$F_K=F_K\mathds 1_{\tilde\O}$}. {According to} the definition of $F_K$ it follows that

\begin{itemize}
\item[(i)] $\int_{\tilde\O} F_k=\int_\O \mathds 1_{\tilde\O} F_K=0$,
\item[(ii)] $\forall k\in \llbracket 0, K\rrbracket$, one has $\int_\O (F_k)\mathds 1_{\tilde\O}\tmm=0$ so that, defining $h_K:=F_K\mathds 1_{\tilde\O}$ we have that $h_K$ is supported in $\tilde \O$, belongs to $L^2(\O)$ and satisfies $\int_\O h_K=0$. Furthermore, the function $\eta_K=-F_K\mathds 1_{\tilde\O}\tmm$ {expands as} %in the basis $\{\psi_k\}_{k\in \N}$ as 
\begin{equation}
\eta_K=\sum_{\ell\geq K+1}\alpha_\ell\psi_\ell\quad \text{ with } \quad \sum_{\ell\geq K+1}\alpha_\ell^2=1.
\end{equation}
\end{itemize}
Finally, we observe that, for this perturbation $h_K$, $\dtmm$ solves
\begin{equation}
\begin{cases}
\mathcal L_m \dtmm=\eta_K, 
\\ \frac{\partial \dtmm}{\partial \nu}=0,
\end{cases}
\end{equation}
whence

\begin{equation}
\dtmm=\sum_{\ell\geq K+1}\frac{\alpha_\ell}{\lambda_\ell(\mathcal L_m)} \psi_\ell.\end{equation}
Using the {$L^2(\O)$-}orthogonality property of the eigenfunctions, we get 
\begin{equation}
\int_\O \dtmm^2=\sum_{\ell\geq K+1}\frac{\alpha_\ell^2}{\lambda_\ell(\mathcal L_m)^2}.
\end{equation}
%In particular, we obtain:
%\begin{equation}\int_\O \dtmm^2=\underset{\ell\to \infty}O\left(\frac1{\lambda_{K+1}^2}\right).\end{equation}
%On the other hand, we have,  then again by the orthogonality of the eigenfunctions, 
{and, similarly,}
\begin{align*}
\mu\int_\O|\n \dtmm|^2-\int_\O (m-2\tmm)\dtmm^2&=\sum_{\ell=K+1}^\infty \frac{\alpha_\ell^2}{\lambda_\ell(\mathcal L_m)}. 
\end{align*}
{We infer the existence of $M>0$ such that }
%As a consequence, we obtain, for some constant $M$, 
\begin{align*}
\mu \int_\O |\n \dtmm|^2&\geq \sum_{\ell=K+1}^\infty \frac{\alpha_\ell^2}{\lambda_\ell(\mathcal L_m)}-M\int_\O \dtmm^2\geq \sum_{\ell=K+1}^\infty \frac{\alpha_\ell^2}{\lambda_\ell(\mathcal L_m)}-M\sum_{\ell=K+1}^\infty \frac{\alpha_\ell^2}{\lambda_\ell(\mathcal L_m)^2}\\
&={\sum_{\ell=K+1}^\infty \frac{\alpha_\ell^2}{\lambda_\ell(\mathcal L_m)^2}(\lambda_\ell(\mathcal L_m)-M)}
\\&\geq \left(\lambda_{K+1}(\mathcal L_m)-M\right)\int_\O \dtmm^2.
\end{align*}

The conclusion follows by taking $K$ large enough {so that $\lambda_{K+1}(\mathcal L_m)>M$, which concludes} the proof.\end{proof}
%\end{proof}

\subsection{Comparison with the results of \cite{NagaharaYanagida}}\label{Se:Comparison}
{This section is dedicated to an} explanation of the main difference with the proof of \cite{NagaharaYanagida}. As recalled in Remark \ref{Re:CompaNY}, the main result of \cite{NagaharaYanagida} {reads:} \textit{if $\tilde \O=\{0<m<1\}$ has an interior point, then it cannot be a solution of Problem~\eqref{Eq:Pv}.} 

Although they do not use the expression \eqref{Eq:DSD} but an alternative {expression} of the second order G\^ateaux-derivative $\ddot F_\mu$, their idea, to reach a contradiction, {is to reason backwards, by finding a function $\psi$, that "should" act as $\dtmm$, well chosen to yield a contradiction, and then constructing an admissible perturbation $h$ supported in the interior of $\{0<m<1\}$ such that $\dtmm=\psi$.}

However, we think that the authors, when {carrying out their reasoning to reach }a contradiction, made a slight mistake in the last line of the second set of equalities on \cite[Page 11]{NagaharaYanagida}. {Indeed, }it seems to us that it is implicitly assumed that if $\dtmm$ is compactly supported in $\tilde \O$, then so is $\ddtmm$, which is in general wrong. 

Nevertheless, we propose {hereafter} an alternative proof of their result that uses their idea of first fixing  a desirable function $\p$, and then {proving the existence of an admissible perturbation $h$, compactly supported in $\tilde \O=\{0<m<1\}$ such that $\p=\dtmm$, leading to} a positive second order derivative.

{Let us} argue by contradiction, { considering a solution $m$ of \eqref{Eq:Pv}} such that the set 
\begin{equation}\tilde\O:=\{0<m<1\}\end{equation} has a non-empty interior (in particular, it is of positive measure). As a consequence of  \eqref{Eq:Opt1} there exists $c$ such that 
\begin{equation}\label{Eq:Opt2}\tmm\pmm=c\text{ in }\tilde \O.\end{equation} Let us pick two interior points $x_0, y_0$ of  $\tilde\O$ and let $r>0$ be such that 
\begin{equation}
\B(x_0;r), \B(y_0;r)\subset \tilde \O, \qquad \B(x_0;r)\cap \B(y;r)=\emptyset.
\end{equation} 
Let $\chi\in \mathcal D(\R^d)$ be a $\mathscr C^\infty$, radially symmetric, non-negative function with compact support in $\B(0;{r})$ such that $\chi(0)=1$. For every $k\in \N$, let us {introduce $\psi_k$ defined by}  
\begin{equation}
\psi_k(x):=\chi(x-x_0)\cos(k|x-x_0|)-\chi(x-y_0)\cos(k|x-y_0|).
\end{equation}
\begin{lemma}\label{Eq:Inter} 
For any $k\in \N$, there exists an admissible perturbation $h_k$ supported in $\tilde \O$, such that 
\begin{equation}
\psi_k=\dtmm[h_k],
\end{equation}
{where $\dtmm[h_k]$ denotes the unique solution of \eqref{LDEdot} associated to the perturbation choice $h=h_k$.}
\end{lemma}
\begin{proof}[Proof of Lemma~\ref{Eq:Inter}]
Let us {introduce $h_k$, defined by}
\begin{equation}
h_{k}:= \frac{1}{\tmm}\left(-\mu \Delta \psi_{k}-(m-2\tmm) \psi_{k}\right).
\end{equation}
Since, by construction $\psi_k\in W^{2,\infty}(\O)$ and since $\inf_\O \tmm>0$ we get that $h_k\in L^\infty(\O)$. {Moreover,} since $\chi$ is compactly supported in $\tilde \O$, {so is $h_k$}. Since $\tilde \O=\{0<m<1\}$, the only condition we have to check to ensure that $h_k$ is admissible at $m$ is that 
\begin{equation}\int_\O h_k=0.\end{equation}
By construction, {one has} 
\begin{equation}
\mu\Delta \psi_k+\psi_k(m-2\tmm)=-h_k\tmm\quad \text{in }\O,
\end{equation} 
so that, {by multiplying this equation by $\pmm$ and integrating twice} by parts we obtain
\begin{equation}\int_\O \psi_k=-\int_\O \tmm \pmm h_k=-c\int_\O h_k\end{equation}
where the last equality comes from \eqref{Eq:Opt2}. Since by construction, $\int_\O \psi_k=0$ the conclusion follows and hence $h_k$ is an admissible perturbation.
\end{proof}
Now, {it remains to prove that}
% that we have proved that $\psi_k$ is a G\^ateaux-derivative, the desired contradiction will be reached provided we can prove that 
\begin{equation}\exists k\in \N^*,\quad \ddot F_\mu(m)[h_k,h_k]= \mu\int_\O \umm |\n \psi_k|^2-\int_\O V_m\psi_k^2>0.\end{equation} Since $\sup_{k\in \N}\Vert \psi_k\Vert_{L^\infty}\leq \Vert \chi\Vert_{L^\infty}$ and since $V_{m,\mu}\in L^\infty(\O)$ {according to \eqref{Eq:V}}, it {is enough} to show that 
\begin{equation}\label{Eq:Goal}
\int_\O \umm |\n \psi_k|^2\xrightarrow[k\to \infty]{} +\infty.
\end{equation}
Using the fact that $\inf_\O \umm>0$ from Estimate \eqref{Eq:U},  \eqref{Eq:Goal} is implied by 
\begin{equation}\label{Eq:Goal2}
\int_\O  |\n \psi_k|^2\xrightarrow[k\to \infty]{}  +\infty.
\end{equation} Finally, since $\mathbb B(x_0;r)\cap \mathbb B(y_0;r)=\emptyset$, \eqref{Eq:Goal2} is in turn implied by

\begin{equation}\label{Eq:Goal3}
\int_{\B(x_0;r)}  |\n \psi_k|^2\xrightarrow[k\to \infty]{}  +\infty.
\end{equation}

Let us now establish \eqref{Eq:Goal3}. By using polar coordinates, {one has}
\begin{align*}
\int_{\B(x_0;r)}|\n \psi_k|^2&=(2\pi)^{d-1}\int_0^r s^{d-1}\left(k\sin(ks)\chi(s)+\cos(ks)\frac{\partial \chi}{\partial s}(s)\right)^2ds
\\&=(2\pi)^{d-1}\int_0^r s^{d-1}k^2\sin^2(ks)\chi(s)\, ds \tag{$I_{1,k}$}\label{Iun}
\\&+2(2\pi)^{d-1}k\int_0^r s^{d-1} \sin(ks)\cos(ks)\chi(s)\frac{\partial \chi}{\partial s}(s)\, ds\tag{$I_{2,k}$}\label{Ideux}
\\&+(2\pi)^{d-1}\int_0^r s^{d-1}\cos^2(ks)\left(\frac{\partial \chi}{\partial s}\right)^2\, ds\tag{$I_{3,k}$}\label{Itrois}.
\end{align*}

Since $\sin^2(k\cdot)$ converges weakly to $\frac12$ {in $L^2(0,r)$}, since $\chi(0)=1$ and $\Vert \chi\Vert_{\mathscr C^1}\leq M$ for some $M>0$, it follows that 
$$
I_{1,k}\underset{k\to +\infty}\sim k^2 C_0, C_0>0\quad\text{and}\quad  I_{2,k}=\underset{k\to \infty}{\operatorname{o}}(I_{1,k}).
$$
Finally, \eqref{Itrois} remains bounded and we get
$$\int_{\B(x_0;r)}|\n \psi_k|^2\underset{k\to \infty}\sim k^2C_0$$ for some constant $C_0>0$, which concludes the proof.

\begin{remark}%[Main difference]
In this approach which, as we underline, {works under the strong hypothesis that $\tilde \O$ has a non-empty interior}, the core point is to build a sequence of admissible perturbations $\{h_k\}_{k\in \N}$ such that the family $\mathscr H=\{h_k\}_{k\in \N}$ is uniformly bounded in $W^{-2,2}$ but not in $W^{-1,2}$; this guarantees the blow-up of the $W^{1,2}$-norm and the boundedness of the $L^2$-norm of the associated G\^ateaux-derivatives $\dot \theta_{m,\mu}[h_k]$. In the proof of Theorem \ref{Th:BangBang}, the perturbation $h$ that we construct has a fixed $L^2$ norm, and hence the sequence of G\^ateaux-derivatives is uniformly bounded in $W^{2,2}(\O)$.
\end{remark}

\subsection{Proof of Theorem \ref{Th:Jcroi}}\label{secproofTh:Jcroi}

The proof of Theorem \ref{Th:Jcroi} follows the same lines as {the one of }Theorem~\ref{Th:BangBang}. For this reason, we only indicate {hereafter} the main steps, and point to the principal differences.
%\begin{proof}[Proof of Theorem \ref{Th:Jcroi}]

Following the same {methodology for stating the first order optimality conditions for problem~\eqref{Eq:Pv}}, let us introduce
%, for a function $j$ satisfying Assumption \eqref{Hyp:J}, 
the adjoint state $p_{j,m,\mu}$ solution of
\begin{equation}\label{Eq:AdjointJ}
\left\{\begin{array}{ll}
\mu \Delta p_{j,m,\mu}+p_{j,m,\mu}(m-2\tmm)=-j'(\tmm) & \text{ in }\O, 
\\ \frac{\partial p_{j,m,\mu}}{\partial \nu}=0 & \text{ on }\partial \O.
\end{array}\right.
\end{equation}
Since $j'>0$, a direct adaptation of Lemma~\ref{Cl:InfP} yields
\begin{equation}\label{Eq:InfPj}
\forall m\in \mathcal M(\O),\quad  \inf_\O p_{j,m,\mu}>0.\end{equation}

It is straightforward to see that the G\^ateaux derivative of the functional $\mathcal J_j$ writes
\begin{equation}
\dot{\mathcal J_j}(m)[h]=\int_\O h \tmm p_{j,m,\mu},
\end{equation}
{for every $m\in \mathcal M(\O)$ and any admissible perturbation $h$ at $m$.}

{Let us} compute the second order G\^ateaux derivative of $\mathcal J_j$. Keeping track of the fact that
%, on the one hand, 
$\ddot{\theta}_{m,\mu}$ solves \eqref{LDEdotdot} and that
%, on the other hand, 
by direct computation, we obtain
\begin{equation}
\ddot{\mathcal J_j}(m)[h,h]=\int_\O \left(\dot\theta_{m,\mu}^2 j''(\tmm)+\ddot\theta_{m,\mu} j'(\tmm)\right),\end{equation} 
we get an expression analogous to \eqref{Eq:DSD}. Indeed, multiplying \eqref{Eq:AdjointJ} by $\ddot{\theta}_{m,\mu}$ and integrating by parts {yields}
\begin{align*}
\frac12\int_\O \ddot{\theta}_{m,\mu} j'(\tmm)&=\int_\O \left(h \dtmm-\dot\theta^2_{m,\mu}\right)p_{j,m,\mu}
\\&=-\int_\O \dot\theta^2_{m,\mu}p_{j,m,\mu}+\int_\O \left(\frac{-\mu \Delta \dtmm-\dtmm(m-2\tmm)}{\tmm}\right)p_{j,m,\mu}.\end{align*} 
{Let us introduce} 
\begin{equation}
u_{j,m,\mu}:=\frac{p_{m,j,\mu}}{\tmm}.
\end{equation} 
{Notice that, using} the same arguments {as in the proof of Theorem~\ref{Th:BangBang},} we obtain 
\begin{equation}\label{Eq:InfJ}
\inf_\O u_{j,m,\mu}>0,\qquad  \Delta u_{j,m,\mu}\in L^\infty(\O).
\end{equation} 
{Since} $j$ {belongs to }$\mathscr C^2$, {there exists} $M_j>0$ {such that} 
\begin{equation}
\Vert j''(\tmm)\Vert_{L^\infty}\leq M_j.
\end{equation}
We {thus} obtain the existence of a potential $V_{j,m,\mu}\in L^\infty(\O)$ such that 
\begin{equation}\label{Eq:DSDj}
\ddot {\mathcal J_j}(m)[h,h]=\mu\int_\O u_{j,m,\mu}\left|\n \dtmm\right|^2-\int_\O V_{j,m,\mu}\dtmm^2.
\end{equation}
As a consequence, by \eqref{Eq:InfJ} and by the fact that $V_{j,m,\mu}\in L^\infty(\O)$, it suffices to find a perturbation $h$ such that, for a large enough parameter $M_0>0$, 
\begin{equation}\int_\O |\n \dtmm|^2\geq M_0\int_\O \dtmm^2.\end{equation}We are now back to proving \eqref{Eq:Fmot}, and the proof reads the same {way}.

\section{Proof of Theorem~\ref{Th:FragND}}
The core of this proof relies on {fine} energy estimates.
%, and the main idea is the same in the proofs of the two Theorems. 
{In what follows, it will be convenient to introduce the set of bang-bang functions 
$$
\mathbb M(\O):=\{m\in \mathcal M(\O), \ \exists E\subset \O\mid m=\mathds 1_E\}.
$$ 
}
To alleviate the reading, let us start with the presentation of the proof structure.

\paragraph{Main idea}
The proof {rests upon the} use of two ingredients:
\begin{itemize}
\item[(i)] the first one {reads}
\begin{lemma}[{\cite[Lemma 2]{MRBSIAP}}]\label{Le:MBR}
There exists $\delta>0$ such that 
\begin{equation}
\underset{\mu \to 0^+}{\lim\inf}\left(\sup_{m\in \mathcal M(\O)}F_\mu(m)\right)\geq m_0+ \delta>0.
\end{equation}
\end{lemma}
\item[(ii)] the second one, on which the emphasis will be put throughout the proof, {is} an estimate of the following form: there exist a constant $M>0$ and two exponents $\alpha, \beta>0$ such that
\begin{equation}\label{Eq:EstimBV}
\forall m\in \mathcal M(\O)\cap BV(\O),\quad  \Vert \tmm-m\Vert_{L^1(\O)}^\beta\leq M \Vert m\Vert_{BV(\O)} \mu^\alpha.
\end{equation} 
\end{itemize}
If Estimate \eqref{Eq:EstimBV} holds, then, assuming that the optimiser $m_\mu^*$ is a $BV(\O)$-function (if it is not, then $\Vert m\Vert_{BV(\O)}=+\infty$ and the statement of the theorem is trivial) then we have 
\begin{equation}\label{m1636}
\mu^\alpha \Vert m_\mu^*\Vert_{BV(\O)}\geq \frac1M\Vert \tmm-m\Vert_{L^1(\O)}^\beta\geq \frac1M \left| \int_\O \tmm-m_0\right|^\beta\geq \frac{\delta^\beta}M,
\end{equation}
{according to Lemma~\ref{Le:MBR},} yielding that 
\begin{equation}
\Vert m_\mu^*\Vert_{BV(\O)}\geq \frac{M'}{\mu^\alpha},
\end{equation}
{with $M'=\frac{\delta^\beta}M$.}
To obtain convergence rates such as \eqref{Eq:EstimBV}, we will proceed using energy arguments and prove that a rescaled, shifted version of the natural energy associated with the PDE \eqref{LDE} yields this kind of control.

\paragraph{The rescaled energy functional}
Let us first recall that the equation \eqref{LDE} admits a variational formulation: {let us} introduce 
$$
\mathcal E_{m,\mu}:W^{1,2}(\O)\ni \theta\mapsto \frac{\mu}2\int_\O |\n \theta|^2+\frac13\int_\O \theta^3-\frac12\int_\O m\theta^2,
$$
then $\tmm$ is characterized as the unique minimiser of $\mathcal E_{\mu}$ over $W^{1,2}(\O)${; in other words}
\begin{equation}\label{Eq:VF}
\mathcal E_{m,\mu}(\tmm)=\inf_{\substack{u\in W^{1,2}(\O)\\ u\geq 0}}\mathcal E_{m,\mu}(u).
\end{equation}

Since we could not locate this formulation in the literature, we give a proof:

\begin{lemma}\label{Le:VF} $\tmm$ is the unique minimiser of 
\begin{equation}\mathcal E_{m,\mu}:u\mapsto \frac\mu2\int_\O |\n u|^2-\frac12\int_\O mu^2+\frac13\int_\O u^3\end{equation}
{over} the set $\mathscr K:=\{u\in W^{1,2}(\O), u\geq 0\text{ in }\O\}$.
\end{lemma}
{For the sake of completeness, this lemma is proved in Appendix~\ref{proofLem:VF} }

{Let us} introduce 
\begin{equation}\tilde{\mathcal E}_{m,\mu}:\{u\in W^{1,2}(\O), u\geq 0\}\ni \theta \mapsto \mathcal E_{m,\mu}(\theta)+\frac16\int_\O m^3.\end{equation}

\begin{remark}
The definition of $\tilde{\mathcal E}_{m,\mu}$ is justified by the following, formal computation: let us assume that $m$ is a $\mathscr C^1$ function. It is known \cite{LouInfluence} that $\tmm\underset{\mu \to 0^+}\rightarrow m$ in $L^p(\O)$, for $p\in [1;+\infty)$. Since {we aim at obtaining} a convergence rate for $\Vert \tmm-m\Vert_{L^1(\O)}$ as $\mu \to 0^+$, it is natural to consider the energy $\mathcal E_{m,\mu}(m)$. Explicit computations show that 
$$\mathcal E_{m,\mu}(m)=\frac{\mu}2\int_\O |\n m|^2-\frac16\int_\O m^3\underset{\mu \to 0}\rightarrow -\frac16\int_\O m^3,$$ {which justifies to} consider the energy $\tilde{\mathcal E}_{m,\mu}$. 
\end{remark}

\paragraph{Estimating $\Vert \tmm-m\Vert_{L^1(\O)}$ using $\tilde{\mathcal E}_{m,\mu}$.}
The  key point is then to prove that $\Vert \tmm-m\Vert_{L^1(\O)}$ {can be estimated in terms of the} rescaled energy, \emph{via} the following two Lemmas.

\begin{lemma}\label{Le:L1}
There exists a constant $M_1>0$ such that 
\begin{equation}
\forall m\in \mathcal M(\O), \quad \Vert \tmm-m\Vert_{L^1(\O)}\leq M_1\tilde{\mathcal E}_{m,\mu}(\tmm)^{\frac13}= M_1\left(\inf_{u\in W^{1,2}(\O), u\geq 0}\tilde{\mathcal E}_{m,\mu}(u)\right)^{\frac13}.
\end{equation}

\end{lemma}
\begin{proof}[Proof of Lemma \ref{Le:L1}]
We {split the proof into two steps.}
\paragraph{Step 1.}
{\it There exists $M>0$ such that 
\begin{equation}
\forall \mu>0, \forall m\in \mathcal M(\O), \qquad \int_\O \left(\frac\tmm3+\frac{m}6\right)(\tmm-m)^2\leq\tilde{\mathcal E}_{m,\mu}(\tmm).
\end{equation}}
%\begin{proof}[Proof of Claim \ref{MGEN}]
This follows from explicit computations. {Setting $A=\int_\O \left(\frac\tmm3+\frac{m}6\right)(\tmm-m)^2$, one has}
\begin{align*}
A &=\frac13\int_\O \tmm \left(\tmm^2-2m\tmm+m^2\right)+\frac16\int_\O m \left(\tmm^2-2m\tmm+m^2\right)
\\=&\frac13\int_\O \tmm^3-\frac23\int_\O m\tmm^2+\frac13\int_\O \tmm m^2+\frac16\int_\O m^3+\frac16\int_\O m\tmm^2-\frac13\int_\O \tmm m^2
\\=&\frac13\int_\O \tmm^3+\frac16\int_\O m^3-\frac12 \int_\O m\tmm^2=\mathcal E_\mu(\tmm)-\frac\mu2\int_\O |\n \tmm|^2+\frac16\int_\O m^3
\\\leq &\mathcal E_\mu(\tmm)+\frac16\int_\O m^3=\tilde{\mathcal E}_\mu(\tmm).
\end{align*}

% We then have the following claim, which immediately yields the conclusion:
\paragraph{Step 2.} 
{\it There exists $M_0>0$ such that for every $\mu>0$ and $m\in {\mathbb M}(\O)$, one has
$$
\Vert \tmm-m\Vert_{L^1(\O)}\leq M_0 \left(\int_\O \left(\frac\tmm3+\frac{m}6\right)(\tmm-m)^{2}\right)^{\frac13}.$$
}
%\begin{proof}[Proof of Claim \ref{Le:Ab}]
According to the Jensen inequality, one has $\Vert \tmm-m\Vert_{L^1(\O)}\leq |\O|^{\frac12}  \left(\int_\O (\tmm-m)^2\right)^{\frac12}$. {Since $m\in {\mathbb M}(\O)$, one has}
\begin{align*}\int_\O \left(\frac\tmm3+\frac{m}6\right)(\tmm-m)^2&= \int_{\{m=0\}}\frac{\tmm^3}3+\int_{\{m=1\}}\frac16(\tmm-1)^2.
\end{align*}
According to the Jensen inequality, there exists a constant $C>0$ such that
\begin{equation}  \int_{\{m=0\}}\frac{\tmm^3}3\geq C \left( \int_{\{m=0\}}{\tmm^2}\right)^{\frac32}.\end{equation}
Furthermore, for any $N_0\in \R_+^*$, there exists $C_{N_0}>0$ such that 
\begin{equation*}
\forall x \in [0;N_0), \quad C_{N_0}x^{\frac32}\leq x.
\end{equation*}
Since {$0\leq |\tmm- m|\leq 2$} a.e. in $\O$, this last inequality yields the existence of $C'>0$ such that
\begin{equation}
\frac16\int_{\{m=1\}} (\tmm-m)^2\geq C' \left(\int_{\{m=1\}} (\tmm-m)^2\right)^{\frac32}
\end{equation}
As a consequence, there exists $C''>0$ such that 
\begin{align*}
\int_\O \left(\frac\tmm2+\frac{m}6\right)(\tmm-m)^2&\geq \int_{\{m=0\}}\frac{\tmm^3}2+\int_{\{m=1\}}{\frac16}(\tmm-m)^2
\\&\geq C'' \left(\int_{\{m=0\}}\tmm^2 \right)^{\frac32}+C''\left(\int_{\{m=1\}} (\tmm-m)^2\right)^{\frac32}
\\&\geq C'' \left(\int_\O (\tmm-m)^2\right)^{\frac32},\end{align*}
whence the conclusion.

{The lemma follows from a combination of the two steps.}  
\end{proof}

As a consequence, we will aim at proving an estimate of the form 
\begin{equation}\label{Eq:kK}
\tilde{\mathcal E}_{m,\mu}(\tmm)\leq \Vert m\Vert_{BV(\O)}\mu^\beta 
\end{equation} 
which, with {Lemma~\ref{Le:L1}, will lead} to an estimate of the type \eqref{Eq:EstimBV}. As explained in the introduction, we provide in Section~\ref{sec:proof1DModica} an alternative proof of \eqref{Eq:kK} in the one-dimensional case following the method of Modica \cite{ModicaMortola}, that cannot unfortunately be straightforwardly extended to higher dimensions.
%\subsection{Proof in the multidimensional case}

Let us now concentrate on the  multidimensional case, assuming that $\O=(0;1)^d$ with $d\geq 1$. We aim at obtaining an estimate of the form \eqref{Eq:EstimBV} which, by Lemma \ref{Le:L1} amounts to determine a bound on $\tilde{\mathcal E}_{m,\mu}(\tmm)$. {Let us first consider $m\in \mathcal M(\O)\cap W^{1,2}(\O)$, that we will} use as a test function in the energy. We get 
\begin{equation}\label{Eq:PL}
\tilde{\mathcal E}_{m,\mu}(\tmm)\leq \tilde{\mathcal E}_{m,\mu}(m)=\frac{\mu}2\int_\O |\n m|^2.\end{equation}

We now consider a convolution kernel defined {with the help of an approximation} of unity. Namely, we consider a $\mathscr C^\infty$  function $\chi$ with compact support in $\mathbb B(0;1)$ satisfying 
$$
0\leq \chi\leq 1\text{ a.e. in }\mathbb B(0;1),\qquad  \int_{\mathbb B(0;1)}\chi=1.
$$ 
For every $\e>0$, we define
\begin{equation}
\chi_\e(x):=\frac1{\e^d}\chi\left(\frac{x}\e\right).
\end{equation} 
%and, for any two functions $f$ and $g$ in $L^1(\R^d)$, their convolution product
%$$f\star g:x\mapsto \int_{\R^d} g(x-y)f(y)dy.$$
{Every $m\in \mathcal M(\O)$ is extended outside of $\O$ by a compactly supported function of bounded variation, according to \cite[Proposition 3.21]{AmbrosioFuscoPallara}.} We define 
\begin{equation}
m_\e:=m\star \chi_\e
\end{equation}
 for every $\e>0$, {where $\star$ stands for the convolution product in $L^1(\R^d)$.}
{It is standard that} 
\begin{equation}
\forall p\in (1;+\infty), \qquad \Vert m-m_\e\Vert_{L^p(\O)}\underset{\e \to 0}\rightarrow 0.
\end{equation}
{According to \cite[Equation~(2.4)]{LouInfluence}}, there exists a constant $M$ such that for every $m,m'\in %L^\infty_+(\O)=
\{f\in L^\infty(\O), \ f\geq 0\}$, there holds 
\begin{equation}
\Vert \tmm-\theta_{m',\mu}\Vert_{L^1(\O)}\leq C\Vert m-m'\Vert_{L^1(\O)}^{\frac13}.\end{equation}
By the triangle inequality, for any $m\in \mathcal M(\O)\cap W^{1,2}(\O)$ and any $\mu,\e>0$,
\begin{align}\label{m2330}
\left|\int_\O \tmm-m_0\right|&\leq \Vert \tmm-\theta_{m_\e,\mu}\Vert_{L^1(\O)}+\Vert  \theta_{m_\e,\mu}-m_\e\Vert_{L^1(\O)}+\left|\int_\O m_\e-m_0\right|.
\end{align}
Note that one has, {for any $i\in \llbracket1,d\rrbracket$,} 
\begin{align*}
\Vert\partial_i m_\e\Vert_{L^2(\O)}^2&=\int_{\O}\left(\frac1{\e^{d+1}}\int_{\R^d}\partial_i \chi\left(\frac{x-y}\e\right)m(y)dy\right)^2\, dx
\\&\leq \frac{M}{\e^{2(d+1)}}\int_{\R^d}\int_\O |\n \chi|^2\left(\frac{x-y}\e\right) |m(y)|\, dydx\quad \text{ by Jensen's inequality}
\\&\leq \frac{M}{\e^{2}}\Vert \n \chi\Vert_{L^2(\R^d)}^2.
\end{align*}
Hence, there exists $M>0$ such that 
$$
\Vert \n m_\e\Vert_{L^2(\O)}\leq \frac{M}\e.
$$

As a consequence, by \eqref{Eq:PL} and Lemma \ref{Le:L1}, we have
\begin{equation}
\Vert \tmm-m\Vert_{L^1(\O)}\leq M_1 \mu^{\frac13}\Vert \n m\Vert_{L^2(\O)}^{\frac23}\leq M_1'\frac{\mu^{\frac13}}{\e^{\frac23}}.
\end{equation}
{It follows that there exists $M>0$} such that
\begin{align*}
\left|\int_\O \tmm-m_0\right|&\leq \Vert \tmm-\theta_{m_\e,\mu}\Vert_{L^1(\O)}+\Vert  \theta_{m_\e,\mu}-m_\e\Vert_{L^1(\O)}+\left|\int_\O m_\e-m_0\right|
\\&\leq M\left(  \Vert m-m_\e\Vert_{L^1(\O)}^{\frac13}+\frac{\mu^{\frac13}}{\e^{\frac23}}+\left|\int_\O m_\e-m_0\right|\right)
\end{align*}
%We estimate each of the two remaining terms appearing separately. 
%First of all, we make use of the following convolution estimate: 
{To end the proof, we need the following lemma, whose proof is postponed to the end of this section for the sake of clarity.}
\begin{lemma}\label{Eq:Tricky}
There exists a constant $d_0>0$ such that
\begin{equation}
\forall m\in \mathcal M(\O), 
\Vert m-m_\e\Vert_{L^1(\O)}\leq d_0 \e\Vert \n m\Vert_{L^1(\O)}.
\end{equation}
\end{lemma}

As a consequence, we have $\Vert m-m_\e\Vert_{L^1(\O)}{\lesssim} \e |m|_{BV(\R^d)}$. Since the extension operator is bounded, we thus get 
$\Vert m-m_\e\Vert_{L^1(\O)}\leq M \e\Vert m\Vert_{BV(\O)}$.
Plugging all these estimates {together and using \eqref{m2330}}, we finally obtain
\begin{align*}
\left|\int_\O \tmm-m_0\right|
%&\leq \Vert \tmm-\theta_{m_\e,\mu}\Vert_{L^1(\O)}+\Vert  \theta_{m_\e,\mu}-m_\e\Vert_{L^1(\O)}+\left|\int_\O m_\e-m_0\right|\\
&\leq M \left(\e \Vert m\Vert_{BV}\right)^{\frac13}+\frac{\mu^{\frac13}}{\e^{\frac23}}+\e\Vert \n m\Vert_{L^1(\O)}
\\&\leq  M \left(\e \Vert m\Vert_{BV}\right)^{\frac13}+\e\Vert \n m\Vert_{L^1(\O)}+M\frac{\mu^{\frac13}}{\e^{\frac23}}.
\end{align*}
{Taking
%, for any $\beta>0$, 
$\e=N_0\mu^{\frac12}$ with $N_0>0$, one gets
$$
\left|\int_\O \tmm-m_0\right|-\frac{M}{N_0^{2/3}}\leq \e \Vert m_\mu^*\Vert_{BV}^{\frac13}+ \e \Vert m_\mu^*\Vert_{BV},
$$
 %Since $\beta>0$, $\frac{\mu^{\frac13}}{\e^{\frac23}}=\mu^{\frac23\beta}\underset{\mu \to 0}\rightarrow 0$.
and} choosing $N_0$ large enough yields 
\begin{equation}\underset{\e\to 0}{\lim\inf} \left(\e \Vert m_\mu^*\Vert_{BV}\right)^{\frac13}+ \left(\e \Vert m_\mu^*\Vert_{BV}\right)>0.
\end{equation}  
{We infer the existence of} $C_0>0$ such that for all solution $m_\mu^*$ of \eqref{Eq:Pv}, one has
\begin{equation}
C_0\leq  \left(\e \Vert m_\mu^*\Vert_{BV}\right)^{\frac13}+ \left(\e \Vert m_\mu^*\Vert_{BV}\right)
\end{equation} 
and therefore, there exists $c_0>0$ such that, for every $\mu>0$ small enough
\begin{equation}
c_0\leq  \e \Vert m_\mu^*\Vert_{BV}={N_0}\mu^{\frac12}\Vert m_\mu^*\Vert_{BV(\O)}.
\end{equation}
The {desired} conclusion follows
%\end{proof}
%{*** j'ai l'impression que la preuve 1D est redondante. Je ne vois pas son apport par rapport à le preuve multi-D. La multi-D contient la 1D, non ? PLUS PRECISE EN 1D****}

\begin{proof}[Proof of Lemma~\ref{Eq:Tricky}]
%We need to estimate $\Vert m-\chi_\e\star m\Vert_{L^1(\O)}.$
As is customary in convolution, one has 
\begin{align*}
m(x)-\chi_\e\star m(x)&\lesssim \underset{|h|=\e}\sup\Vert \tau_h m-m\Vert_{L^1(\R^d)}
\end{align*}
for a.e. $x\in \O$, where $\tau_h$ stands for the translation operator. However, we claim that
\begin{equation}\label{Eq:Central}
\sup_{|h|=\e}\Vert \tau_h m-m\Vert_{L^1(\R^d)}\leq \Vert h\Vert_{L^\infty}{\Vert \n m\Vert_{L^1(\R^d)}}.
\end{equation}
%\begin{proof}[Proof of \eqref{Eq:Central}]
{It suffices} to prove \eqref{Eq:Central} for $m\in \mathscr C^1$, and the general result follows from the density of $\mathscr C^1$ functions in $BV(\R^d)$. For any $h\in \R^d$ we have
\begin{align*}
\int_{\R^d}\left| m(x+h)-m(x)\right|dx&=\int_{\R^d}\left|\int_0^1 \frac{d}{d\xi}[m(x{+}\xi h)]\, d\xi\right|\, dx
=\int_{\R^d}\left|\int_0^1 \langle \n m(x{+}\xi h),h\rangle\right|\, dx
\\&\leq \Vert h\Vert_{\infty} \int_0^1\int_{\R^d} |\n m(x{+}\xi h)|d\xi \, dx
\\&=\Vert h\Vert_\infty \int_{\R^d}|\n m|=\Vert h\Vert_\infty \Vert \n m\Vert_{L^1}.
\end{align*}
{The desired result follows.}
\end{proof}

\section{Conclusion: possible extensions of the bang-bang property to other state equations}\label{Se:BC}We conclude this article with a discussion on possible generalisations of our method. Indeed, an interesting question is to know whether or not the methods put forth in the proof of Theorem \ref{Th:BangBang} could be applied to other types of boundary conditions, for instance Dirichlet or Robin, or for other {kinds of }non-linearities. We { justify below that} it is the case, and that the main difficulty lies in the well-posedness of the equation acting as a constraint on the optimisation problem \eqref{Eq:Pv}. 
%In order to ease the reading, we will make the assumptions on $F$ (assumptions \eqref{Eq:HConc}-\eqref{Eq:1c}-\eqref{Eq:BL} below) when they become necessary.

Let us consider a boundary operator $\mathcal B$, that may be of Neumann ($\mathcal Bu=\frac{\partial u}{\partial \nu}$) or of Robin type ($\mathcal Bu=\frac{\partial u}{\partial \nu}+\beta u$ for some $\beta>0$). Let us consider a non-linearity $F=F(x,u)$ of class $\mathscr C^2$, and consider, for a given $m\in \mathcal M(\O)$, the solution $u_m$ of 
\begin{equation}\label{Eq:Concl}
\left\{\begin{array}{ll}
-\Delta u_m=mu_m+F(x,u_m) & \text{ in }\O,\\ 
\mathcal B u_m=0 & \text{ on }\partial \O.\end{array}\right.
\end{equation}

The first assumption on $F$ one has to make is:
\begin{equation}\tag{$\bold H$}\label{Eq:HConc}
\begin{array}{l}
\text{For any $m\in \mathcal M(\O)$, \eqref{Eq:Concl} has a unique positive solution $u_m$.}\\
\text{Furthermore, $\displaystyle \inf_{m\in \mathcal M(\O)}\inf_\O u_m>0, \sup_{m\in \mathcal M(\O)}\Vert u_m\Vert_{\mathscr C^{1}}<\infty.$}
\end{array}
\end{equation} 
It is notable that \eqref{Eq:HConc} is satsfied whenever $F$ satisfies:
\begin{enumerate} 
\item $F(x,0)=0$ and the steady state $z(\cdot)= 0$ is unstable.
\item Uniformly w.r.t. $x\in \O$, one has $\displaystyle \lim_{y\to \infty}F(x,y)/y=-\infty$.
\end{enumerate}

This is for instance problematic when considering Dirichlet boundary conditions for the logistic-diffusive equation: depending on the range of $\mu$ the equation may only have trivial solutions.

We also assume 
\begin{equation}\tag{$\bold H'$}\label{Eq:1c}
\text{The map $m\mapsto u_m$ is twice G\^ateaux-differentiable.}
\end{equation}
which is for instance ensured whenever $F\in W^{2,\infty}$ and if the solution $u=u_m$ is linearly stable.

Consider then the following {optimisation} problem, where the function $j$ satisfies \eqref{Hyp:J} 
%(i.e. it is $\mathscr C^2$ and increasing on a sufficiently large interval)
\begin{equation}\label{Eq:PvConcl}\tag{$\bold P$}
\sup_{m\in \mathcal M(\O)}J(m),\qquad {\text{with }J(m)=\int_\O j(u_m)}.
\end{equation}
%Then we believe that 
Our methods enable us to prove that any solution of \eqref{Eq:PvConcl} is a bang-bang function. To do so, we need to write down the first and second order G\^ateaux-derivative of $u_m$ with respect to $m$: {using the same notations as in the rest of this article,} if \eqref{Eq:HConc}-\eqref{Eq:1c} hold, then it can be shown that 
%the first and second order G\^ateaux-derivatives of $u_m$ with respect to $m$ solve
\begin{equation}
\left\{\begin{array}{ll}
-\Delta \dot u_m-m\dot u_m-\dot u_m \frac{\partial F}{\partial u}(x,u_m)=hu_m & \text{ in }\O, 
\\ \mathcal B\dot u_m=0 & \text{ on }\partial \O, 
\end{array}\right.
\end{equation} 
and 
\begin{equation} 
\left\{\begin{array}{ll}
-\Delta \ddot u_m-m\ddot u_m-\ddot u_m \frac{\partial F}{\partial u}(x,u_m)=2h\dot u_m+(\dot u_m)^2\frac{\partial^2F}{\partial u^2}(x,u_m) & \text{ in }\O, 
\\ \mathcal B\ddot u_m=0 & \text{ on }\partial \O, 
\end{array}\right. 
\end{equation}
{This allows us to compute the derivative of $J$. Under} \eqref{Eq:HConc}-\eqref{Eq:1c}, we have
\begin{equation}
\dot J(m)[h]=\int_\O \dot u_mj'(u_m)\quad \text{ and }\quad \ddot J(m)[h,h]=\int_\O \left(\dot u_m\right)^2j''(u_m)+\int_\O \ddot u_m j'(u_m).
\end{equation}
Let us introduce the adjoint state $p_m$, solving
\begin{equation}
\left\{\begin{array}{ll}
-\Delta p_m-mp_m-p_m \frac{\partial F}{\partial u}(x,u_m)=j'(u_m) & \text{ in }\O, 
\\ \mathcal Bp_m=0 & \text{ on }\partial \O.
\end{array}\right.\end{equation}  
Here, we need to make another assumption on $F$:
\begin{equation}\label{Eq:BL}\tag{$\bold H''$}
\text{ For any $m\in \mathcal M(\O)$, \quad $\inf_\O p_m>0$.}
\end{equation}
Given the assumption on $j$, \eqref{Eq:BL} is for instance implied if the first eigenvalue of $-\Delta-m-\frac{\partial F}{\partial u}(\cdot,u_m)$ is positive (linear stability condition), that also ensures the G\^ateaux differentiability of $m\mapsto u_m$. This in turn holds if $F(x,u)=-ug(x,u)$ with $g{\in W^{2,\infty}}$ non-decreasing. {Using the adjoint state to compute $\ddot J(m)[h,h]$ in a more tractable form}, one has 
\begin{align*}
\ddot J(m)[h,h]=&\int_\O (\dot u_m)^2j''(u_m)+\int_\O \ddot u_m j'(u_m)
\\=&\int_\O (\dot u_m)^2j''(u_m)+\int_\O p_m\left(2h\dot u_m+(\dot u_m)^2\frac{\partial^2F}{\partial u^2}(x,u_m)\right)
\\=&\int_\O (\dot u_m)^2\left(j''(u_m)+\frac{\partial^2 F}{\partial u^2}(x,u_m)\right)+2\int_\O \left(\frac{p_m}{u_m}\right)\dot u_m\left(-\Delta \dot u_m-m\dot u_m-\dot u_m \frac{\partial F}{\partial u}(x,u_m)\right)
\\=&\int_\O (\dot u_m)^2\left(j''(u_m)+\frac{\partial^2 F}{\partial u^2}(x,u_m)-m\frac{p_m}{u_m}-\frac{p_m}{u_m}\frac{\partial F}{\partial u}(x,u_m)\right)
\\&+2\int_\O \frac{p_m}{u_m}(-\dot u_m\Delta \dot u_m).
\end{align*}
Let us set $\Psi_m:=\frac{p_m}{u_m}$ and $V_m:=j''(u_m)+\frac{\partial^2 F}{\partial u^2}(x,u_m)-m\Psi_m-\Psi_m\frac{\partial F}{\partial u}(x,u_m)$. Then we obtain
\begin{align*}
\ddot J(m)[h,h]=&\int_\O (\dot u_m)^2V_m+2\int_\O\dot u_m\langle \n \Psi_m,\n\dot u_m\rangle+2\int_\O \Psi_m|\n \dot u_m|^2
\\=&\int_\O \left(\dot u_m\right)^2 (V_m-\Delta \Psi_m)+2\int_\O \Psi_m|\n \dot u_m|^2.
\end{align*}
Defining $V_m'=V_m-\Delta \Psi_m$, we finally get 
\begin{equation}\ddot J(m)[h,h]=\int_\O 2\Psi_m|\n \dot u_m|^2+\int_\O (\dot u_m)^2V_m'\end{equation} and the assumptions we made on $F$ allow us to conclude that $V_m'$ {belongs to} $L^\infty$ and that $\inf_\O \Psi_m>0$. 
To obtain the bang-bang property, {we} argue by contradiction and assume that 
%a solution $m^*$ of \eqref{Eq:PvConcl} is not bang-bang or, in other words, that 
the set $\tilde \O:=\{0<m^*<1\}$ is of positive measure. To reach a contradiction, it suffices to exhibit a perturbation $h$ that is supported in $\tilde \O$ such that 
\begin{equation}
\int_\O h=0\quad \text{and}\quad \int_\O |\n \dot u_m|^2>\frac{\Vert V_m'\Vert_{L^\infty(\O)}}{\inf_\O \Psi_m}\int_\O \left(\dot u_m\right)^2.
\end{equation} 
{Following the proof} of Theorem~\ref{Th:BangBang}, we introduce the sequence of eigenfunctions and eigenvalues $\{\p_k, \lambda_k\}_{k\in \N}$ associated to the operator 
\begin{equation}\mathcal L_m:=-\Delta-\left(m+\frac{\partial F}{\partial u}(x,u_m)\right)\end{equation}  with $\mathcal B\p_k=0$. {Adapting hence} the proof of Theorem~\ref{Th:BangBang}, we show that for any $K\in \N$, there exists an admissible perturbation $h$ such that
%, in this basis, we have 
\begin{equation}
hu_m=\sum_{k\geq K} \alpha_k \p_k,\qquad  \sum_{k\geq K}\alpha_k^2=1.
\end{equation} 
It follows that for such a perturbation, 
\begin{equation}
\int_\O |\n \dot u_m|^2\geq \lambda_K\int_\O \left(\dot u_m\right)^2.
\end{equation} 
{Choosing $K\in \N$ large enough so that $\lambda_K\geq \frac{\Vert V_m'\Vert_{L^\infty(\O)}}{\inf_\O \Psi_m}$ yields the expected conclusion.}
\newpage
\appendix
\begin{center}
\LARGE{\bf Appendix}
\end{center}
\section{Proof of Lemma \ref{Le:VF}}\label{proofLem:VF}
Let us first recall that since $\tmm$ {is non-negative and does not vanish in $\O$}, we have 
\begin{equation}
\mathcal E_{m,\mu}(\tmm)=-\frac16\int_\O \tmm^3<0,
\end{equation}
so that $u(\cdot)= 0$ is not a minimiser of $\mathcal E_{m,\mu}$.

In order to prove this Lemma, let us introduce the energy functional
\begin{equation}
\mathscr F_{m,\mu}:W^{1,2}(\O)\ni u\mapsto\frac\mu2\int_\O |\n u|^2-\frac12\int_\O mu^2+\frac13\int_\O |u|^3.
\end{equation}
 {Observe that}
 \begin{equation}\label{Eq:Int}
 \forall u\in W^{1,2}(\O),\qquad  \mathscr F_{m,\mu}(u)=\mathscr F_{m,\mu}(|u|)=\mathcal E_{m,\mu}\left(|u|\right).
 \end{equation}
 In particular, if $\mathscr F_{m,\mu}$ has a minimiser $u^*$, then $|u^*|$ also minimises $\mathscr F_{m,\mu}$, and $|u^*|$ solves \begin{equation}\inf_{u\in \mathscr K}\mathcal E _{m,\mu}(u).\end{equation} Conversely, if $u_*\geq 0$ is a minimiser of $\mathcal E_{m,\mu}$ then for any $z\in W^{1,2}(\O)$, 
 \begin{equation}\mathscr F_{m,\mu}(z)=\mathscr F_{m,\mu}(|z|)=\mathcal E_{m,\mu}(|z|)\geq \mathcal E_{m,\mu}(u_*)=\mathscr F_{m,\mu}(u_*)\end{equation} and so $u_*$ is a minimiser of $\mathscr F_{m,\mu}$.
 
 Let us then prove that $\tmm$ is a minimiser of $\mathscr F_{m,\mu}$. 
Consider a minimising sequence $\{y_k\}_{k\in \N}$ of $\mathscr F_{m,\mu}$. Up to replacing $y_k$ with $|y_k|$ which, thanks to \eqref{Eq:Int}, would still yield a minimising sequence, {we can assume that for every $k\in \N$, $y_k$ is non-negative}.
% \begin{equation}
% \forall k\in \N, y_k\geq 0.
% \end{equation}
% 
 Let us introduce $\lambda(m)$ as the first eigenvalue of the operator $-\Delta -m$ with Neumann boundary conditions. {According to the Courant-Fischer principle, one has}
 %The variational formulation 
 \begin{equation}
 \lambda(m)=\inf_{\substack{u\in W^{1,2}(\O)\\ \int_\O u^2=1}}\left(\mu \int_\O |\n u|^2-\int_\O mu^2\right)\end{equation} 
 {and therefore} 
 \begin{equation}
 \mathscr F_{m,\mu}(y_k)\geq \lambda(m)\Vert y_k\Vert_{L^2(\O)}^2+\frac13\Vert y_k\Vert_{L^3(\O)}^3.
 \end{equation}
 
% We also know there exists a constant $C>0$ such that
% \begin{equation}\label{Eq:Jensen}
%\Vert y_k\Vert_{L^2(\O)}^2\leq C\Vert y_k\Vert_{L^3(\O)}^2\end{equation} 
 {Since the embedding $L^3(\O)\hookrightarrow L^2(\O)$ is continuous, there exists $C>0$ such that}
 %and it follows that 
 \begin{equation}
 \sup_{k\in \N}\left(C \lambda(m)\Vert y_k\Vert_{L^3(\O)}^2+\frac13\Vert y_k\Vert_{L^3(\O)}^3\right)<\infty.
 \end{equation} 
 As a consequence, $\{y_k\}_{k\in \N}$ is bounded in $L^3(\O)$ {and then} also in $L^2$ {by using the same argument}. Finally, {by} definition of $\mathscr F_{m,\mu}$, it is also uniformly bounded in $W^{1,2}(\O)$. 
 
Hence, there exists a strong $L^2(\O)$, {weak $L^3(\O)$ and} weak $W^{1,2}$ closure point $y_\infty \in W^{1,2}(\O)$  of $\{y_k\}_{k\in \N}$. 
%There also exists a weak $L^3$ closure point $z_\infty$ of this sequence. This implies $z_\infty=y_\infty$.
Since the map $\R\ni x\mapsto |x|^3$ is convex, the map $L^3(\O)\ni y\mapsto \int_\O |y|^3$ 
 is lower semi-continuous. Hence it follows that
 \begin{equation}
 \liminf_{k\to \infty}\mathscr F_{m,\mu}(y_k)\geq \mathscr F_{m,\mu}(y_\infty),
 \end{equation} 
{and $y_\infty$ minimises $\mathscr F_{m,\mu}$ over $\mathscr K$}. {Since $0_{L^2(\O)}$} is not a minimiser, we have $y_\infty\geq 0$ and $y_\infty(\cdot)\neq 0$.
 
 The map $x\mapsto |x|^3$ is $\mathscr C^1$ and the Euler-Lagrange equation on $y_\infty$ writes
 \begin{equation}
 \left\{\begin{array}{ll}
 \Delta y_\infty+y_\infty(m-y_\infty)=0 & \text{in }\O, \\ 
 \frac{\partial y_\infty}{\partial \nu}=0 & \text{in }\partial\O,\\
 y_\infty\geq 0.
 \end{array}\right.
 \end{equation} 
 From uniqueness for non-zero, non-negative solutions of the logistic-diffusive PDE, {it follows} that $y_\infty=\tmm$. As a consequence:
 \begin{equation}\mathcal E_{m,\mu}(\tmm)=\mathscr F_{m,\mu}(\tmm)=\min_{ W^{1,2}(\O)}\mathscr F_{m,\mu}=\min_{\mathscr K}\mathcal E_{m,\mu},\end{equation}
 which concludes the proof.

\section{Proof of \eqref{Eq:kK} in the one-dimensional case}\label{sec:proof1DModica}
{We assume in this section that $\O=(0,1)$}. Let us prove \eqref{Eq:kK}.
The proof relies on ideas {by} Modica \cite{ModicaMortola}. Given Lemma \ref{Le:MBR} {and \eqref{m1636}}, it is enough to establish a uniform convergence rate of $\tmm$ to $m$ in $L^1(\O)$ {with respect to the $BV(\O)$ norm of $m$, as} $\mu \to 0$ . 

We proceed in several steps, first considering the case where $m$ is the characteristic function of a set of finite perimeter before encompassing the general case. Let us {introduce} 
\begin{equation}\mathcal M_M(\O):=\left\{m\in \mathcal M(\O), \Vert m\Vert_{BV(\O)}\leq M\right\}
\quad \text{and}\quad
%\end{equation} 
%as well as 
%\begin{equation}
{\mathbb M}_M(\O):= {\mathbb M}(\O)\cap \mathcal M_M(\O),
\end{equation}
{for every $M>0$.}
{The following Proposition is the key point of the proof.}
\begin{proposition}\label{Pr:1D}
There exists $C_1>0$ such that
\begin{equation}
\forall M>0, \forall m\in {\mathbb M}(\O), \qquad \tilde{\mathcal E}_{m,\mu}(\tmm)\leq C_1\sqrt{\mu}\Vert m\Vert_{BV(\O)}.
\end{equation}
\end{proposition}

We can now prove Theorem 1. First of all, the maximiser $m_\mu^*$ of \eqref{Eq:Pv} is a bang-bang function by Theorem \ref{Th:BangBang} {and belongs therefore to $\mathbb{M}(\O)$}. We thus obtain 
\begin{equation}
\delta^3\leq C_1^3 \sqrt{\mu}\Vert m_\mu^*\Vert_{BV(\O)},
\end{equation}
{where $\delta>0$ is given by Lemma \ref{Le:MBR}}. The conclusion follows.

\begin{proof}[Proof of Proposition \ref{Pr:1D}]
In what follows, we will bypass the distinction between the interior perimeter of a subset $A\subset (0;1)$, denoted $\operatorname{Per}_{int}(A)$, and its perimeter denoted $\operatorname{Per}(A)$ when seen as a subset of $\R$. Since we have obviously
\begin{equation}
 \operatorname{Per}_{int}(A)\leq \operatorname{Per}(A)\leq \operatorname{Per}_{int}(A)+2,
\end{equation}
it follows that there exists $c_0>0$ such that, for any set $A$ of finite perimeter
\begin{equation}
\operatorname{Per}_{int}(A)\geq 2\quad \Rightarrow\quad  c_0 \operatorname{Per}(A)\leq  \operatorname{Per}_{int}(A)\leq \operatorname{Per}(A).
\end{equation} 
Furthermore, since we know from \cite{MRBSIAP} that the $BV(\O)$ norm of maximisers blows-up as $\mu \to0$, we can always assume that the set of finite perimeter $A$ we are working with satisfies $\operatorname{Per}_{int}(A)\geq 2.$ 
%{*** je ne suis pas s\^ur de bien voir l'int\'er\^et de la remarque ci-dessus. ***}

Since $m\in {\mathbb M}_M(\O)$, we know that $m$ writes $m=\mathds 1_A$ {where $A$ is a set of bounded perimeter}.

Let us then consider such a subset $A$.  Since $A$ is of finite perimeter, it writes
\begin{equation}
A=\bigsqcup_{i=1}^n (a_i;b_i)
\end{equation} 
with $0\leq a_i<b_i<a_{i+1}\leq1$ for every $i\in \llbracket1,\dots,n\rrbracket$.

To obtain the conclusion of the Proposition, it suffices to exhibit a constant $C_1$ that does not depend on $\mu,m$, and a function $u_\mu \in W^{1,2}(\O)$ such that
\begin{equation}
\tilde{\mathcal E}_{m,\mu}(u_\mu)\leq  C_1\sqrt{\mu}\operatorname{Per(A)}.
\end{equation}

{Let us introduce $h_A$, the so-called signed-distance} function {to the set $A$, defined by} 
\begin{equation}
h_A:x\mapsto \begin{cases} \operatorname{dist}(x,\partial A)\text{ if }x \notin A, \\0\text{ if } x \in \partial A, \\ -\operatorname{dist}(x,\partial A)\text{ if }x \in A,
\end{cases}
\end{equation} 
as well as the auxiliary function
\begin{equation}
\phi_\e:\R\ni t\mapsto \left\{\begin{array}{ll} 
1 &  \text{ if }t<0, \\ 
0 & \text{ it }t\geq \eta_\e, \\ 
1-\frac{t}{\eta_\e} & \text{ otherwise},
\end{array}\right.\end{equation} 
for some regularization parameter $\e\geq 0$, where $\eta_\e=\e^{\frac14}$.
We combine these two functions and introduce $u_\e= \phi_\e\circ h_A$. Let us use $u_\e$ as a test function in the variational formulation \eqref{Eq:VF}. We will estimate separately the gradient term and the {remainder term of the energy functional}. 

\paragraph{Estimate of the gradient term.} Since $h_A$ is differentiable a.e. and $|h'_A|=1$, we have $(u_\e')^2=\phi_\e'(h_A(\cdot))^2$ a.e. in $\O$. 
%Hence, 
%\begin{equation}
%\int_0^1(u_\e')^2=\int_0^1\phi_\e'(h_A(t))^2dt.
%\end{equation} 
%We now split the integral 
Using the decomposition of $A$, we get
\begin{equation}
\int_0^1(u_\e')^2=\int_0^{a_1}\phi_\e'(h_A(t))^2dt+\sum_{i=1}^n\left\{ \int_{a_i}^{b_i}\phi_\e'(h_A(t))^2dt+\int_{b_i}^{a_{i+1}}\phi_\e'(h_A(t))^2dt\right\}+\int_{a_{n+1}}^1\phi_\e'(h_A(t))^2dt.
\end{equation} 
Let us focus on the term 
$$
\sum_{i=1}^n\left\{ \int_{a_i}^{b_i}\phi_\e'(h_A(t))^2\, dt+\int_{b_i}^{a_{i+1}}\phi_\e'(h_A(t))^2\, dt\right\}.
$$ 
The main interest of this decomposition is that on each interval $(a_i;b_i)$ or $(b_i;a_{i+1})$, the function $h$ is symmetric with respect to the midpoint of the interval.  As a consequence, two cases may occur {when} considering the interval $(a_i;b_i)$ (the case $(b_i;a_{i+1})$ being exactly identical):
\begin{itemize}
\item[(i)] either $|b_i-a_i|\leq 2\eta_\e$, in which case, since $\Vert \phi_\e'\Vert_{L^\infty}=\frac1{\eta_\e}$ it follows that 
\begin{equation}
\int_{a_i}^{b_i}\phi_\e'(h_A(t))^2dt\leq 2\eta_\e \Vert \phi_\e'\Vert_{L^\infty}^2\leq \frac2{\eta_\e}.
\end{equation}
\item[(ii)] or $|b_i-a_i|>2\eta_\e$, in which case $|\{u_\e'\neq 0\}\cap (a_i;b_i)|\leq 2\eta_\e$ and so 
\begin{equation}
\int_{a_i}^{b_i}\phi_\e'(h_A(t))^2dt\leq 2\eta_\e \Vert \phi_\e'\Vert_{L^\infty}^2\leq \frac2{\eta_\e}.
\end{equation}
\end{itemize}
As such, we have 
\begin{equation}\sum_{i=1}^n\left\{ \int_{a_i}^{b_i}\phi_\e'(h_A(t))^2dt+\int_{b_i}^{a_{i+1}}\phi_\e'(h_A(t))^2dt\right\}\leq \frac{4n}{\eta_\e}\leq 2\frac{\operatorname{Per}(A)}{\eta_\e}.
\end{equation}

The end terms 
$$\int_0^{a_1}\phi_\e'(h_A(t))^2dt+\int_{b_n}^1\phi_\e'(h_A(t))^2dt$$ are handled in the same way, and we finally obtain 
\begin{equation}\label{m1931}
\int_0^1 (u_\e')^2(t)dt\leq C \frac{\operatorname{Per}(A)}{\eta_\e}
\end{equation}
for some constant $C>0$.

\paragraph{Estimate of the potential term.} 
It remains to {deal with} the quantity
\begin{equation}\frac13\int_0^1 u_\e^3(t)dt-\frac12\int_0^1 mu_\e^2(t)dt+\frac16\int_0^1 m^3.
\end{equation}
If we define $\psi_\e=\frac13u_\e^3-\frac12 mu_\e^2+\frac16 m^3$ we have the following decomposition: in the set $\{m=1\}$, we have $h_A\leq 0$, hence $u_\e=1$ and {we infer that }$\psi_\e= 0$ in $\{m=1\}$.

The integral to estimate boils down to 
\begin{equation}
\int_0^1 \psi_\e(t)\mathds 1_{\{m=0\}}\, dt=\int_0^1\frac13 u_\e^3\mathds 1_{\{m=0\}}.
\end{equation}
However, we can do exactly the same distinction as for the analysis of the gradient part of the energy: for any $i\in \llbracket 1,n\rrbracket$ (the end intervals are handled in the same way) we either have $|a_{i+1}-b_i|\leq 2\eta_\e$, in which case 
\begin{equation} \int_{b_i}^{a_{i+1}}\psi_\e(t)dt\leq 2\eta_\e\end{equation} or $|a_{i+1}-b_i|>2\eta_\e$, in which case the same conclusion holds {since $0\leq \psi_\e\leq 1$ a.e. in $\O$}. As a consequence, we obtain 
\begin{equation}\label{m1930}
\int_0^1 \psi_\e(t)\leq 2\eta_\e \operatorname{Per}(A).
\end{equation}

{Combining \eqref{m1931} and \eqref{m1930}} yields the existence of $C_1>0$ such that 
\begin{equation}
\tilde{\mathcal E}_{m,\mu}(u_\e)\leq C_1\left(\frac\mu{\eta_\e}+\eta_\e\right)\operatorname{Per}(A).
\end{equation}
{Picking} $\eta_\e=\sqrt{\mu}$, we obtain 
\begin{equation}
\tilde{\mathcal E}_{m,\mu}(u_\e)\leq2 C_1\sqrt{\mu}\operatorname{Per}(A),
\end{equation}
{leading to the desired conclusion}.
\end{proof}

\bibliographystyle{abbrv}

\bibliography{BiblioFrag}

\end{document}